\documentclass{article}
\usepackage[utf8]{inputenc}
\usepackage{fullpage}
\usepackage{authblk}
\usepackage{amsmath,amsfonts,amssymb,amsthm,enumerate,mdframed,hyperref,cleveref,algorithmic,algorithm,multirow,import,tikz,booktabs,xcolor,soul}
\usetikzlibrary{arrows,positioning, shapes} 
\usepackage{caption,subcaption}

\newtheorem{theorem}{Theorem}
\newtheorem{proposition}[theorem]{Proposition}
\newtheorem{lemma}[theorem]{Lemma}

\newtheorem{definition}[theorem]{Definition}
\newtheorem{corollary}[theorem]{Corollary}

\mathcode`*=\numexpr\mathcode`*-"2000\relax

\definecolor{darkgreen}{rgb}{0,0.6,0}

\providecommand{\F}{\mathbb{F}}
\providecommand{\N}{\mathbb{N}}

\DeclareMathOperator{\PG}{PG}

\newcommand{\cB}{\mathcal{B}}

\newcommand{\cH}{\mathcal{H}}
\newcommand{\cL}{\mathcal{L}}
\newcommand{\cM}{\mathcal{M}}
\newcommand{\cP}{\mathcal{P}}

\newcommand{\cS}{\mathcal{S}}

\title{Vector space partitions of $\operatorname{GF}(2)^8$}
\author{Sascha Kurz}
\affil{University of Bayreuth, 95440 Bayreuth, Germany\\sascha.kurz@uni-bayreuth.de}

\date{February 2022}

\begin{document}

\maketitle

\maketitle

\begin{abstract}
  \noindent
  A vector space partition $\cP$ of the projective space $\PG(v-1,q)$ is a set of subspaces in $\PG(v-1,q)$ which partitions the set of points. 
  We say that a vector space partition $\cP$ has type $(v-1)^{m_{v-1}} \dots 2^{m_2}1^{m_1}$ if precisely $m_i$ of its elements have dimension $i$, where 
  $1\le i\le v-1$. Here we determine all possible types of vector space partitions in $\PG(7,2)$.

  \smallskip
  
\end{abstract}

\section{Introduction}
\label{sec_introduction}
We call the $i$-dimensional subspaces of the projective space $\PG(v-1,q)$ $i$-spaces using the geometric names points, lines, planes, solids, and hyperplanes for $1$-, $2$-, 
$3$-, $4$-, and $(v-1)$-spaces, respectively. A \emph{vector space partition} $\cP$ of $\PG(v-1,q)$ is a set of subspaces in $\PG(v-1,q)$ which partitions the set of points. For 
a survey on known results we refer to \cite{heden2012survey}. We say that a vector space partition $\cP$ has type $(v-1)^{m_{v-1}} \dots 2^{m_2}1^{m_1}$ if precisely $m_i$ of its 
elements have dimension $i$, where $1\le i\le v-1$. The classification of the possible types of a vector space partition, given the parameters $v$ and $q$, is an important and 
difficult problem. Based on \cite{heden1986partitions}, the classification for the binary case $q=2$ was completed for $v\le 7$ in \cite{el2009partitions}. Under the assumption 
$m_1=0$ the case $(q,v)=(2,8)$ has been treated in \cite{el2010partitions}. Here we complete the classification of the possible types of vector space partitions in $\PG(7,2)$. 
We will also briefly discuss the feasible types of vector space partitions in $\PG(v-1,q)$ for all field sizes $q$ and all dimensions $v\le 5$.

Setting $[k]_q:=\tfrac{q^k-1}{q-1}$ we can state that every $k$-space in $\PG(v-1,q)$ consists of $[k]_q$ points. So, counting points gives the 
\emph{packing condition}
\begin{equation}
  \sum_{i=1}^{v-1} m_i\cdot[i]_q =[v]_q. \label{eq_packing_condition}
\end{equation} 
Another well-known condition uses the fact that an $a$-space and a disjoint $b$-space span an $(a+b)$-space. So, we have 
\begin{equation}
  m_i\cdot m_j=0  \label{eq_dimension_condition}
\end{equation}
for all $1\le i<j\le v-1$ with $i+j>j$ and $m_i\le 1$ for all $i>v/2$. These two conditions are sufficient to characterize all feasible types of vector space partitions in $\PG(v-1,q)$ 
for $v\le 4$. 

Another condition stems from the fact that for an index $2\le j\le v-1$ with $m_j>0$ the set of points contained in the subspaces of dimension strictly less than $j$ corresponds to
a $q^{j-1}$-divisible linear code over $\F_q$ of length
\begin{equation}  
  n=\sum_{i=1}^{j-1} m_i\cdot[i]_q,
\end{equation}
see e.g.\ \cite{kiermaier2020lengths}, where a linear code is called $\Delta$-divisible if all of its codewords have a weight that is divisible by $\Delta$. Non-existence results 
for $q^r$-divisible projective codes are e.g.\ discussed in \cite{honold2018partial}. A recent survey can be found in \cite{kurz2021divisible}. Some of these conditions 
are already contained in \cite{heden2009length} and used under the name \emph{tail condition} in the literature on vector space partitions.

A few more necessary conditions for the existence of vector space partitions are stated in \cite{lehmann2012some}.

\medskip

The notion of a vector space partition can be generalized in several directions. A $\lambda$-fold vector space partition of $\PG(v-1,q)$ is a (multi-) set of subspaces such that 
every point $P$ is covered exactly $\lambda$ times, see e.g.\ \cite{el2011lambda}. Another variant considers sets of subspaces such that every $t$-space is covered exactly once, 
see \cite{heinlein2019generalized}. Vector space partitions of affine spaces have been considered in \cite{bamberg2022affine}.

\bigskip

The remaining part of the paper is structured as follows. In Section~\ref{sec_preliminaries} we introduce the necessary preliminaries. We deduce our main result -- the classification 
of all possible types of vector space partitions of $\PG(7,2)$ -- in Section~\ref{sec_vsp_pg_7_2}. While several of the presented non-existence results for vector space partitions 
have purely theoretical proofs, others rely on extensive computer computations. It would be nice to see some of these calculations be replaced by pen-and-paper proofs. To top 
the paper off, we discuss the possible types of vector space partitions in $\PG(v-1,q)$ for arbitrary field sizes $q$ and dimensions $v\le 5$ of the ambient space in
Section~\ref{sec_vsp_pg_small}.

\section{Preliminaries}
\label{sec_preliminaries}

A vector space partition $\cP$ of $\PG(v-1,q)$ is called \emph{reducible} if there exists a proper subset $\mathcal{Q}$ of the elements of $\cP$ whose points partition 
a proper subspace of $\PG(v-1,q)$. If $\cP$ is not reducible we also speak of an irreducible vector space partition. We can easily construct reducible vector space partitions 
by starting from an arbitrary vector space partition and replacing an element with dimension at least $2$ by its contained points. 

\bigskip

A \emph{multiset of points} in $\PG(v-1,q)$ is a mapping $\chi$ from the set of points in $\PG(v-1,q)$ to $\mathbb{N}$, so that $\chi(P)$ is the multiplicity of the point $P$. If $\chi(P)\le 
1$ for all points $P$, then we also speak of a set (of points) $\chi$ in $\PG(v-1,q)$. The \emph{support} of a multiset of points is the set of all points with non-zero multiplicity. We 
say that $\chi$ is $\Delta$-divisible if $\sum_{P\not\le H} \chi(P)\equiv 0\pmod \Delta$ for each hyperplane $H$ of $\PG(v-1,q)$ and points $P$. In other words, the Hamming weights of the codewords 
of the $\F_q$-linear code $C_\chi$ associated with $\chi$ are divisible by $\Delta$. By $\#\cM$ we denote the cardinality of $\cM$, i.e., the sum $\sum_P \cM(P)$ of the multiplicities 
of all points $P$. If $S$ is an arbitrary subspace, then by $\cM(S)$ we denote the sum $\sum_{P\le S} \cM(P)$ of the point multiplicities of the points contained in $S$. For each  
$U\in\operatorname{PG}(v-1,q)$ we denote by $\chi_U$ its characteristic function, i.e., 
$\chi_U(P)=1$ iff $P\le U$ and $\chi_U(P)=0$ otherwise. It is an easy exercise to show that $\chi_U$ is $q^{\dim(U)-1}$-divisible, which extends to multisets of subspaces:  
\begin{lemma}(\cite[Lemma 11]{kiermaier2020lengths}) Let $\mathcal{U}$ be a multiset of subspaces in $\operatorname{PG}(v-1,q)$ with dimension at least $k$. Then $\chi_{\mathcal{U}}:= 
  \sum_{U\in\mathcal{U}} \chi_U$ is $q^{k-1}$-divisible.
\end{lemma}
If $\chi$ is $\Delta$-divisible and $\chi(P)\le \lambda$ for some constant $\lambda\in \mathbb{N}$ and all points $P$, then the $\lambda$-complement $\overline{\chi}$, defined by 
$\overline{\chi}(P)=\lambda-\chi(P)$ for all points $P$, is also $\Delta$-divisible.
\begin{corollary}
  Let $\cP$ be a vector space partition of $\PG(v-1,q)$ of type $(v-1)^{m_{v-1}} \dots 2^{m_2}1^{m_1}$, then the set of 
  points $\cH_k$ such that the corresponding element $A\in\cP$, that contains the point, has dimension at most $k$ is $q^k$-divisible if $\cP$ contains an element with 
  dimension strictly larger than $k$.
\end{corollary}

An exemplary implication is that no vector space partition of type $3^8 2^2 1^1$ in $\PG(5,2)$ exists since there is no $2$-divisible multiset of points of cardinality $1$, i.e., no $2$-divisible 
binary code of effective length $1$. In our setting of vector space partitions the maximum possible point multiplicity is $1$, so that the corresponding codes are projective, i.e., generator matrices 
do not contain repeated columns. The possible effective lengths of projective binary $\Delta$-divisible linear codes have been completely characterized for all 
$\Delta\in\{2,4,8\}$, see e.g.\ \cite{honold2018partial,honold2019lengths} for proofs and further references:
\begin{proposition}
  \label{prop_lenghts_div_proj}
  Let $n\in\mathbb{N}_{>0}$ be the effective length of a $\Delta$-divisible binary projective linear code.
  \begin{enumerate}
    \item[(a)] If $\Delta=2$, then $n\ge 3$.
    \item[(b)] If $\Delta=4$, then $n\in \{7,8\}$ or $n\ge 14$.
    \item[(c)] If $\Delta=8$, then $n\in\{15, 16, 30, 31, 32, 45, 46, 47, 48, 49, 50, 51\}$ or $n\ge 60$.
  \end{enumerate}
  All those effective lengths can indeed be attained.   
\end{proposition}

\begin{definition}
  We say that a vector space partition $\cP$ has an \emph{$s$-supertail of type} $a_1^{m_1} a_2^{m_2}\dots a_s^{m_s}$, where 
  $a_1>a_2>\dots a_s\ge 1$, if $\cP$ contains exactly $m_j$ elements of dimension $a_j$ for all $1\le j\le s$, there exists at 
  least one element $A\in\cP$ with $\dim(A)>a_1$, and for all elements $B\in\cP$ with $\dim(B)\le a_1$ there exists an index 
  $1\le j\le s$ with $\dim(B)=a_j$.
\end{definition}
For the ease of notation we also allow the choice of $m_j=0$ and just speak of a supertail of a certain type. From Proposition~\ref{prop_lenghts_div_proj}  
we can directly conclude that certain types of supertails are impossible:
\begin{corollary}
  Let $\cP$ be a vector space partition of $\PG(v-1,2)$, then $\cP$ cannot have a supertail of one of the following types:
  \begin{itemize}
    \item $1^1$, $1^2$;
    \item $2^0 1^i$ for $i\in \{1,\dots,6,9,\dots,13\}$, $2^1 1^i$ for $i\in\{0,\dots, 3,6,\dots,10\}$, $2^2 1^i$ for $i\in \{0,3,\dots,7\}$, 
          $2^3 1^i$ for $i\in \{0,\dots,4\}$, $2^4 1^i$ for $i\in \{0,1\}$; and
    \item $3^a 2^b 1^c$ where $7a+3b+c<60$ and $7a+3b+c\notin\{0,15, 16, 30, 31, 32, 45, 46, 47, 48, 49, 50, 51\}$.        
  \end{itemize}
\end{corollary}
For literature on the supertail we refer e.g.\ to \cite{heden2013supertail,nuastase2018complete}.

\medskip

For small $n$ the projective $\Delta$-divisible $\F_q$-linear codes of effective length $n$ can be exhaustively generated with 
software packages like e.g.\ \texttt{Q-Extension} \cite{bouyukliev2007q} or \texttt{LinCode} \cite{bouyukliev2021computer}. 
Having the point sets at hand, we can check whether they can be partitioned into a certain number of planes, lines, and points, which 
excludes a few further supertail types. E.g.\ one can easily show that each $q^2$-divisible multiset of cardinality $q^2+q+1$ over
$\F_q$ has to be the characteristic function of a plane, so that there is no supertail of type $2^2 1^{q^2-q-1}$ over $\F_q$.
For enumeration results of projective binary $\Delta$-divisible codes for $\Delta\in\{2,4,8\}$ we refer to \cite{ubt_eref40887}.   
Since the corresponding codes are computationally shown to be unique we have:
\begin{lemma}
  \label{lemma_q_2_subspaces}
  Let $\cS$ be a $\Delta$-divisible set of cardinality $n$ over $\F_2$. 
  \begin{itemize}
    \item[(a)] If $(\Delta,n)=(2,3)$, then $\cS$ is the characteristic function of a line.
    \item[(b)] If $(\Delta,n)=(4,7)$, then $\cS$ is the characteristic function of a plane.
    \item[(c)] If $(\Delta,n)=(4,14)$, then $\cS$ is the sum of the characteristic functions of two disjoint planes.
    \item[(d)] If $(\Delta,n)=(8,15)$, then $\cS$ is the characteristic function of a solid.
    \item[(e)] If $(\Delta,n)=(8,30)$, then $\cS$ is the sum of the characteristic functions of two disjoint solids.
  \end{itemize}  
\end{lemma}
Theoretical proofs and generalizations can e.g. be found in \cite[Section 11]{kurz2021divisible} or \cite{honold2018partial}. For $q>2$ a much stronger 
result is known. \cite[Theorem 13]{govaerts2003particular} directly implies:
\begin{theorem}
  Let $\cM$ be a $q^r$-divisible multiset of cardinality $\delta\cdot\tfrac{q^{r+1}-1}{q-1}$ over $\F_q$, where $r\in\N_{>0}$. If $q > 2$ and $1 \le\delta <
  \varepsilon$, where $q+\varepsilon$ is the size of the smallest non-trivial blocking sets in $\PG(2, q)$, then there exist $(r+1)$-spaces $S_1,\dots,S_\delta$ 
  such that
  $$
    \cM = \sum_{i=1}^\delta \chi_{S_i},
  $$
  i.e., $\cM$ is the sum of $(r+1)$-spaces.
\end{theorem}
Using dimension arguments, c.f.\ the proof of Lemma~\ref{lemma_three_solids}, we conclude from Lemma~\ref{lemma_q_2_subspaces}.(c)-(e):  
\begin{corollary}
  Let $\cP$ be a vector space partition of $\PG(v-1,2)$, then $\cP$ cannot have a supertail of one of the following types:
  \begin{itemize}
    \item $2^3 1^5$; 
    \item $3^1 2^1 1^5$; and 
    \item $3^2 2^1 1^{13}$, $3^2 2^2 1^{10}$, $3^2 2^3 1^{7}$, $3^2 2^4 1^{4}$, $3^1 2^6 1^5$.  
  \end{itemize}
\end{corollary}
Lemma~\ref{lemma_q_2_subspaces}.(a) yields the fact that every vector space partition of $\PG(v-1,2)$ of type $(v-1)^{m_{v-1}}\dots 2^{m_2} 1^3$ is reducible (assuming $v>2$), i.e., 
the three points have to form a line. Similarly, from Lemma~\ref{lemma_q_2_subspaces}.(b) we an conclude that every vector space partition of $\PG(v-1,2)$ of type $(v-1)^{m_{v-1}}\dots 
3^{m_3} 1^7$ is reducible (assuming $v>3$), i.e., the seven points have to form a plane.
 
While Lemma~\ref{lemma_q_2_subspaces} discusses the parameters of (repeated) simplex codes, i.e., duals of Hamming codes, there are also known uniqueness results for 
the parameters of first order Reed-Muller codes, see e.g.\ \cite{kurz2021generalization}. Translated to geometry we have:
\begin{lemma}
  Let $\cS$ be a $2^r$-divisible set of $2^{r+1}$ points in $\PG(v-1,2)$, where $r\in\mathbb{N}_{>0}$. Then, $\cS$ is the characteristic function of an
  affine $(r+1)$-space.
\end{lemma} 
In our context an important implication is that such a set $\cS$ does not contain a line, so that we conclude:
\begin{lemma}
  Let $\cP$ be a vector space partition of $\PG(v-1,2)$, then $\cP$ cannot have a supertail of one of the following types:
  \begin{itemize}
    \item $2^1 1^5$ and 
    \item $3^0 2^1 1^{13}$, $3^0 2^2 1^{10}$, $3^0 2^3 1^7$, $3^0 2^4 1^4$.   
  \end{itemize}
\end{lemma}
For larger field sizes there exist examples different to affine subspaces. They can be obtained by the so-called cylinder construction, see e.g.\ \cite{de2019cylinder}, and share the 
property that those point sets also do not contain a line. A few parameters where $q^r$-divisible sets of $q^{r+1}$ points in $\PG(n-1,q)$ have to be obtained by the cylinder construction 
are recently discussed in \cite{kurz2021generalization}.

\bigskip

We say that a multiset of points $\cM$ in $\PG(v-1,q)$ is \emph{spanning} if the points $P$ with non-zero multiplicity $\cM(P)>0$ span the entire ambient space. For a given multiset 
$\cM$ of points in $\PG(v-1,q)$ we denote by $a_i$ the number of hyperplanes $H$ such that $\cM(H)=i$. The vector $\left(a_i\right)_{i\in\mathbb{N}}$ is called the \emph{spectrum} of 
$\cM$. If $\cM$ is spanning, then we have $a_{\#\cM}=0$. By considering $\cM$ 
restricted to $K\cong \PG(k-1,q)$ we can always assume that $\cM$ is spanning if we choose a suitable integer for $k$. For the ease of notation we 
assume that $\cM$ is spanning in $\PG(k-1,q)$ in the following. Counting the number of hyperplanes in $\PG(k-1,q)$ gives   
\begin{equation}
  \sum_{i} a_i=\frac{q^k-1}{q-1}\label{se1}
\end{equation} 
and counting the number of pairs of points and hyperplanes gives
\begin{equation}
  \sum_{i} ia_i=\#\cM\cdot \frac{q^{k-1}-1}{q-1}.\label{se2}
\end{equation} 
For the third equation we assume $\cM(P)\in\{0,1\}$ for every point $P$, i.e., there is no point with multiplicity at least $2$.  
Double-counting the incidences between pairs of elements in $\cM$ and hyperplanes gives

\begin{equation}
  \sum_{i} {i \choose 2} a_i={{\#\cM}\choose 2}\cdot \frac{q^{k-2}-1}{q-1}.\label{se5}
\end{equation}
We call the equations (\ref{se1})-(\ref{se5}) the \emph{standard equations} for sets of points.  

\section{Vector space partitions in $\PG(7,2)$}
\label{sec_vsp_pg_7_2} 

For each dimension $v$ of the ambient space and each field size $q$ the packing condition in Equation~(\ref{eq_packing_condition}) combined with 
the dimension condition in Equation~(\ref{eq_dimension_condition}) leave over a finite list of possibly types of vector space partitions in $\PG(v-1,q)$. 
The conditions on the supertail mentioned in the previous section eliminate a few more cases. Here we will treat the case of $\PG(7,2)$, which is 
quite comprehensive compared to the case $\PG(6,2)$. In the end it will turn out that there are more than ten thousand different possible types of    
vector space partitions in $\PG(7,2)$. E.g.\ there are vector space partitions of type $4^4 3^1 2^{61-i} 1^{5+3i}$ for each integer $0\le i\le 61$. 
Of course it is sufficient to give constructions for the {\lq\lq}irreducible cases{\rq\rq} only. I.e., in our example it suffice to give an 
example of type $4^4 3^1 2^{61} 1^{5}$ and then to replace $i$ lines by its three contained points each. We have utilized the following general 
reduction rules:
\begin{lemma}
  Let $\cP$ be a vector space partition of type $(v-1)^{m_{v-1}} \dots 2^{m_2}1^{m_1}$ in $\PG(v-1,2)$. 
  \begin{itemize}
    \item If $m_2>0$, then there also exists a vector space partition of type $(v-1)^{m_{v-1}} \dots 3^{m_3} 2^{m_2-1}1^{m_1+3}$.
    \item If $m_3>0$, then there also exists a vector space partition of type $(v-1)^{m_{v-1}} \dots 4^{m_4} 3^{m_3-1} 2^{m_2}1^{m_1+7}$.
    \item If $m_3>0$, then there also exists a vector space partition of type $(v-1)^{m_{v-1}} \dots 4^{m_4} 3^{m_3-1} 2^{m_2+1}1^{m_1+4}$.
    \item If $m_4>0$, then there also exists a vector space partition of type $(v-1)^{m_{v-1}} \dots 5^{m_5} 4^{m_4-1} 3^{m_3} 2^{m_2}1^{m_1+15}$.
    \item If $m_4>0$, then there also exists a vector space partition of type $(v-1)^{m_{v-1}} \dots 5^{m_5} 4^{m_4-1} 3^{m_3} 2^{m_2+5}1^{m_1}$.
    \item If $m_4>0$, then there also exists a vector space partition of type $(v-1)^{m_{v-1}} \dots 5^{m_5} 4^{m_4-1} 3^{m_3+1} 2^{m_2}1^{m_1+8}$.
  \end{itemize}
\end{lemma} 
Nevertheless, still a lot of examples of vector space partitions need to be constructed. To this end we utilize integer linear programming (ILP) formulations. 
For each $i$-space $S$ we introduce the binary variable $x_S^i\in\{0,1\}$ with the meaning $x_S^i=1$ iff $S$ is contained in the vector space partition $\cP$ in $\PG(v-1,q)$. 
The partitioning condition is modeled by
\begin{equation}
  \sum_{i=1}^{v-1} \sum_{S\,:\,\dim(S)=i, P\le S} x_S^i = 1
\end{equation} 
for each point $P$. Introducing the counting variables
\begin{equation}
  m_i= \sum_{S\,:\,\dim(S)=i} x_S^i
\end{equation}
for each dimension $1\le i\le v-1$, we can prescribe the type or maximizing/minimizing certain values $m_i$ while prescribing the others. Of course this general ILP 
has a huge number of variables and constraints. E.g.\ there are $200\,787$ solids in $\PG(7,2)$. Luckily enough we can use the collineation group of $\PG(7,2)$ to 
reduce the search space a bit. It is well known that the collineation group acts transitively on $i$-spaces as well as pairs of disjoint $i$- and $j$-spaces, including the case $i=j$. 
For triples of disjoint $i$-spaces $A$, $B$, $C$ the orbits have $\dim(\langle A,B,C\rangle)$ as invariant that can vary between $2i$ and $\min\{v,3i\}$, where $v=8$ in 
our situation. So, we can prescribe up to three elements of the vector space partition and search for the others. 

In the special case where $m_4$ is relatively large, we have considered a Desarguesian spread for solids in $\PG(7,2)$, see e.g.\ \cite{lunardon1999normal}. So, let $\mathcal{S}$ be 
such a set of $17$ pairwise disjoint solids in $\PG(7,2)$ (that form a vector space partition of type $4^{17}$). We can restrict the search space for the ILP by 
forcing the used solids to be contained in $\mathcal{S}$, i.e., we set $x_S^4=0$ for every solid $S$ that is not contained in $\mathcal{S}$. Another starting point 
for vector space partitions are so-called lifted MRD (maximum rank distance) codes, see e.g.\ \cite{sheekey201913}, that e.g.\ give vector space partitions of types $6^1 2^{64}$, 
$5^1 3^{32}$, and also $4^{17}$. 

For ILP formulations of search problems in incidence structures there is a general (heuristic) method to reduce the search space -- the so-called Kramer--Mesner method \cite{kramer1976t}. 
Here a suitable subgroup $G$ of the automorphism group of the desired object is assumed. The action of the group $G$ partitions the variables into orbits and we assume that 
variables in the same orbit have the same value. I.e., in our context the vector space partition consists of entire orbits of solids, planes, and so on. This approach reduces the 
number of variables in terms of the order of $G$. Also the number of constraints is automatically reduced since usually several constraints become identical if we use orbit 
variables. For a more detailed and precise description of the method we refer e.g.\ to \cite{kohnert2008construction} where sets of subspaces in $\PG(v-1,q)$ with restricted 
intersection dimension were considered. Via this method we can, of course, find examples whose automorphism group contains $G$ as a subgroup only. So, as mentioned, the method 
is heuristic, but quite successfully applied in the literature. In our context we have used groups $G$ with orders up to $16$. 

As our setting is quite comprehensive we refrain from stating the explicit details for the cases. However, explicit vector space partitions for specific types in $\PG(7,2)$
can be obtained from the author upon request. While it took quite some time to compute an example for each feasible case, we now focus on the complementary non-existence results.

\bigskip

\begin{lemma}
  Let $\cP$ be a vector space partition of $\PG(v-1,2)$, then $\cP$ cannot have a supertail of type $2^4 1^5$.
\end{lemma}
\begin{proof}
  There are three different $4$-divisible sets of cardinality $17$, one for each of the dimensions $k\in \{6,7,8\}$, see \cite{ubt_eref40887}. Only the $6$-dimensional 
  point set admits three disjoint lines and there cannot be four disjoint lines.
\end{proof}
As a consequence, there does not exist a vector space partition of type $3^{34} 2^4 1^5$ of $\PG(7,2)$ but there exist vector space partitions of type $3^{34} 2^{3-i} 1^{8+3i}$ 
for all integers $0\le i\le 3$. A set of $34$ pairwise disjoint planes in $\PG(7,2)$ was constructed for the first time in \cite{el2010maximum}. Now several thousand non-isomorphic 
examples are known, see \cite{honold2018partial}.  

\begin{lemma}
  \label{lemma_meeting_lines}
  For $k\ge 3$ let $K_1,K_2,K_3$ be three pairwise disjoint $k$-spaces in $\PG(v-1,2)$ and $\cL$ be the set of lines that intersect each $K_i$ in exactly a point. Then, the lines in $\cL$ 
  are pairwise disjoint and $\left\{ L\cap K_i\,:\, L\in\cL\right\}$ forms a subspace for each index $1\le i\le 3$. 
\end{lemma}
\begin{proof}
  First we show that the lines in $\cL$ are pairwise disjoint. To the contrary we assume that that $\cL$ contains two lines $L_1=\left\langle A,B\right\rangle$ and 
  $L_2\left\langle A,B'\right\rangle$, where we assume w.l.o.g.\ that $A\le K_1$ and $B,B'\le K_2$. Then $A+B,A+B'\le K_3$, so that $B+B'\in K_2\cap K_3$, which is a contradiction. 
  
  For the second part we show the statement for $K_1$ and assume that $L_1=\left\langle A_1,B_1\right\rangle$ and $L_2=\left\langle A_2,B_2\right\rangle$ are two different lines 
  in $\cL$ with $A_1,A_2\le K_1$ and $B_1,B_2\le K_2$. Then also $\left\langle A_1+A_2,B_1+B_2\right\rangle$ is in $\cL$ since $A_1+A_2\le K_1$, $B_1+B_2\le K_2$, and 
  $A_1+A_2+B_1+B_2\le K_3$.  
\end{proof}

\begin{lemma}
  For $k\ge 3$ let $K_1,K_2,K_3$ be three pairwise disjoint $k$-spaces in $\PG(2k-1,2)$ and $\cL$ be the set of lines that intersect each $K_i$ in exactly a point. Then, 
  $\#\cL=2^k-1$ and the pairwise disjoint lines in $\cL$ partition the point set $K_1\cup K_2\cup K_3$.
\end{lemma}
\begin{proof}
  We have already mentioned that there is a unique isomorphism type under the operation of the collineation group of $\PG(2k-1,2)$ since $\dim(\left\langle K_i, K_j\right\rangle)=2k$ 
  for all $1\le i<j\le 3$. W.l.o.g. we choose $K_1=\left\langle e_1,\dots,e_k\right\rangle$, $K_2=\left\langle e_{k+1},\dots,e_{2k}\right\rangle$, 
  and $K_3=\left\langle e_1+e_{k+1},\dots,e_k+e_{2k}\right\rangle$, where $e_i$ denotes the $i$th unit vector. Let $P_1=\langle \sum_{i=1}^k a_ie_i\rangle$ be the unique point of such a line 
  in $K_1$ and $P_2=\langle \sum_{i=1}^k b_ie_{i+k}\rangle$ be the unique intersection of the line with $K_2$, where $a_1,\dots,a_k,b_1,\dots,b_k\in\{0,1\}$. Since the third 
  point of the line is given by $P_3=\sum_{i=1}^k \left(a_ie_i+b_ie_{i+k}\right)\in K_3$, we conclude $a_i=b_i$ for all $1\le i\le k$. Moreover, we have to exclude the case 
  $a_1=\dots=a_k$, so that $2^k-1$ possibilities remain.
\end{proof}

\begin{lemma}
  \label{lemma_three_solids}
  Let $\cP$ be a vector space partition of $\PG(v-1,2)$, then $\cP$ cannot have a supertail of one of the following types: $3^4 2^4 1^5$, $3^4 2^3 1^8$, $3^4 2^2 1^{11}$, 
  $3^4 2^1 1^{14}$, $3^4 2^0 1^{17}$, and $3^1 2^{11} 1^5$.
\end{lemma}
\begin{proof}
  The stated types for supertails all consist of $45$ points. The $8$-divisible sets $\cS$ of cardinality $45$ in $\PG(v-1,2)$ have been computationally classified in \cite{ubt_eref40887}. 
  Either $\cS$ is the sum of the characteristic functions of three pairwise disjoint solids or $\cS$ is isomorphic to the points $P_1,\dots, P_9$ of a projective base in $\PG(7,2)$ 
  and the ${9\choose 2}=36$ remaining points on the lines $P_iP_j$. 
  
  We can easily check that the latter case does not contain a plane in its support. Thus, we can assume that $\cS=\chi_{S_1}+\chi_{S_2}+\chi_{S_3}$, where $S_1,S_2,S_3$ are pairwise 
  disjoint solids. Each plane $E$ contained in the support of $\cS$ has to intersect each of the solids $S_i$ in a $d_i$-dimensional subspace where $d_i\in\{0,1,2,3\}$. Since two lines contained 
  in a plane have to intersect in a point, the only possibility is that one of the solids $S_i$ completely contains the plane $E$. However, this is impossible for four pairwise disjoint planes.
  So, let us finally assume that $\cS$ has type $3^1 2^{11} 1^5$ and that the unique plane $E$ is contained in $S_1$. The five points form a $2$-divisible set, so that they have to be isomorphic 
  to a projective base $\cB$ of size $5$, see e.g.\ \cite{ubt_eref40887}. So, at most four of these five points can be contained in $S_1$. Denoting the set of lines that 
  intersect $S_1$, $S_2$, and $S_3$ in exactly a point by $\cL$ and setting $\alpha:=\#\cL$, we observe $\alpha\ge 4$. So, $S_2$ and $S_3$ can contain at most three lines each and $\alpha\ge 5$. 
  Since $\cB$ is $2$-divisible we have $\alpha\neq 8$, so that $5\le \alpha\le 7$.
  
  Note that no three points in $\cP_1:=\left\{L\cap S_1\, :\, L\in\cL\right\}$ can form a line, so that Lemma~\ref{lemma_meeting_lines} yields that also 
  no three points in $\cP_i:=\left\{L\cap S_i\, :\, L\in\cL\right\}$ can form a line for all $1\le i\le 3$. Thus the $\alpha\ge 5$ points in $\cP_3$ have to span a solid that is 
  also contained in $\left\langle S_1,S_2\right\rangle$, so that we conclude $n=8$ for the dimension of the ambient space. Note that for $i=2,3$ the points of $S_i$ are partitioned by 
  the lines contained in $S_i$ and the points $\cP_i\cup \left(\cB\cap S_i\right)$. Note that the points in $\cP_i\cup \left(\cB\cap S_i\right)$ form a $2$-divisible set in 
  $\PG(3,2)$ whose cardinality is a multiple of three. For cardinality three we can apply Lemma~\ref{lemma_q_2_subspaces} to conclude that those points form a line. 
  If $\#\left(\cP_i\cup \left(\cB\cap S_i\right)\right)=3\beta$, where $2\le \beta\le 5$, then there exists a unique $2$-divisible set of $3\beta$ points in $\PG(3,2)$, see
  \cite{ubt_eref40887}. In all cases the span of the points has dimension $3$ and an example is given by $\beta$ disjoint lines. Thus, for $i=2,3$ the points in 
  $\cP_i\cup \left(\cB\cap S_i\right)$ can be partitioned into lines, so that $\alpha=\#\cP_i\le 2\beta =\frac{2}{3} \cdot \left(\alpha+\#\left(\cB\cap S_i\right)\right)$, which 
  is equivalent to $\alpha\le 2\#\left(\cB\cap S_i\right)$. Since $\#\left(\cB\cap S_2\right)+\#\left(\cB\cap S_3\right) = \#\cB=5$ and $\alpha\ge 5$ this is impossible.  
\end{proof}
So, especially there does not exist a vector space partition of type $4^{14} 3^4 1^{17}$ or $4^{14} 3^1 2^{11} 1^5$ in $\PG(7,2)$. (We remark that we have also checked the more involved 
second part of the proof of Lemma~\ref{lemma_three_solids}, i.e., that the set of points of three disjoint solids cannot be partitioned in a plane, eleven lines, and five points 
computationally using an ILP formulation.)

\begin{lemma}
  \label{lemma_four_solids_parameters}
  In $\PG(7,2)$ no vector space partitions of types 
  $4^{13} 3^6 2^{6-i} 1^{3i}$ for $0\le i\le 6$, types $4^{13} 3^5 2^{7-i} 1^{4+3i}$ for $0\le i\le 7$, and type 
  $4^{13} 3^4 2^{9} 1^{5}$ exist.  
\end{lemma}
\begin{proof}
  Assume that $\cP$ is a vector space partition of one of those types and observe that the set $\cH$ of points that are not covered by the $13$ solids of $\cP$ forms 
  an $8$-divisible set of $60$ points. Let $k$ be the dimension of the span of $\cH$. For the ease of notation we assume that $\cH$ is embedded in $\PG(k-1,2)$ and denote by 
  $a_i$ the number of hyperplanes containing exactly $i$ points from $\cH$, so that the standard equations are given by 
  \begin{eqnarray}
    \sum_{i=0}^6 a_{4+8i}&=& 2^k-1\label{ase1},\\
    \sum_{i=0}^6 (4+8i)\cdot a_{4+8i}&=& 60\cdot\left(2^{k-1}-1\right), \text{ and}\label{ase2}\\
    \sum_{i=0}^6 (4+8i)(3+8i)\cdot a_{4+8i}= 60\cdot 59\cdot\left(2^{k-2}-1\right)\label{ase3}.
  \end{eqnarray} 
  Using Equation~(\ref{ase1}) and Equation~(\ref{ase2})  
  we conclude
  \begin{equation}  
    \sum_{i=0}^6 i \cdot a_{4+8i}= 13\cdot 2^{k-2}-7\label{ase4}
  \end{equation}  
  Equation~(\ref{ase3}) minus $56$ times Equation~(\ref{ase2}) minus $12$ times Equation~(\ref{ase1}) gives 
  \begin{equation}  
    \sum_{i=0}^6 i^2 \cdot a_{4+8i}= 691\cdot 2^{k-6}-49\label{ase5}
  \end{equation} 
  after a final division by $64$. Now, Equation~(\ref{ase5}) minus seven times Equation~(\ref{ase4}) plus twelve times Equation~(\ref{ase1}) yields
  \begin{equation}
    \sum_{i=0}^6 (i-3)(i-4)a_{4+8i} =3\cdot 2^{k-6}-12.
  \end{equation} 
  Since $(i-3)(i-4)\ge 0$ and $a_{4+8i}\ge 0$ for all $i$, we have $k\ge 8$. Due to the ambient space $\PG(7,2)$ for $\cP$ we are only interested in the case $k=8$ 
  where $\sum_{i=0}^7 (i-3)(i-4)a_{4+8i}=0$. Thus, the unique solution is given by $a_{36}=60$, $a_{28}=195$, and $a_i=0$ otherwise.
  
  We have used \texttt{LinCode}, see \cite{bouyukliev2021computer}, to enumerate all projective $[60,8,\{24,32\}]_2$ codes. There are exactly $12$ non-isomorphic such 
  codes, c.f.~\cite{bouyukliev2006projective}. We have computationally checked that the corresponding sets of points contain at most five disjoint planes. Five disjoint 
  planes can be obtained in just one of the twelve cases. However, we have computationally checked that in this case no $13$ pairwise disjoint solids can be chosen
  outside this point set. Four planes and nine lines that are pairwise disjoint can occur in two of the twelve cases. Again, we have computationally checked that in 
  those two cases no $13$ pairwise disjoint solids can be chosen outside the point set.
\end{proof}

\begin{lemma}
   \label{lemma_five_solids_parameters_aux}
  If $\cP$ is a vector space partition of type $4^{12} 3^8  1^{19}$ in $\PG(7,2)$, then the point set of the points in the subspaces of dimension at most $3$ 
  is isomorphic to the columns of 
  $$
  \left(\begin{smallmatrix}
  1 1 1 1 1 1 1 1 1 1 1 1 1 1 1 1 1 1 1 1 1 1 1 1 1 1 1 1 1 1 1 0 0 0 0 0 0 0 0 0 0 0 0 0 0 0 0 0 0 0 0 0 0 0 0 0 0 0 0 0 0 0 0 0 0 0 0 1 0 0 0 0 0 0 0 \\ 
  0 0 0 0 0 0 0 0 0 0 0 0 0 0 0 1 1 1 1 1 1 1 1 1 1 1 1 1 1 1 1 1 1 1 1 1 1 1 1 1 1 1 1 1 1 1 0 0 0 0 0 0 0 0 0 0 0 0 0 0 0 0 0 0 0 0 0 0 1 0 0 0 0 0 0 \\
  0 0 0 0 0 0 0 1 1 1 1 1 1 1 1 0 0 0 0 0 0 0 0 1 1 1 1 1 1 1 1 0 0 0 0 0 0 0 1 1 1 1 1 1 1 1 1 1 1 1 1 1 1 0 0 0 0 0 0 0 0 0 0 0 0 0 0 0 0 1 0 0 0 0 0 \\ 
  0 0 0 1 1 1 1 0 0 0 0 1 1 1 1 0 0 0 0 1 1 1 1 0 0 0 0 1 1 1 1 0 0 0 1 1 1 1 0 0 0 0 1 1 1 1 0 0 0 1 1 1 1 1 1 1 0 0 0 0 0 0 0 0 0 0 0 0 0 0 1 0 0 0 0 \\ 
  0 0 1 0 0 1 1 0 0 1 1 0 0 0 1 0 0 0 1 0 0 1 1 0 0 0 0 0 1 1 1 0 1 1 0 0 0 1 0 0 0 1 0 0 1 1 0 0 1 0 0 1 1 1 1 1 1 1 1 1 1 1 1 0 0 0 0 0 0 0 0 1 0 0 0 \\ 
  0 1 0 0 1 0 1 0 0 0 1 0 0 0 0 0 0 1 0 0 1 0 1 0 1 1 1 1 0 0 1 1 0 1 1 1 1 0 0 0 1 1 0 1 0 0 0 1 1 0 1 0 0 0 0 1 0 0 0 1 1 1 1 1 1 1 0 0 0 0 0 0 1 0 0 \\ 
  1 1 0 1 1 1 1 0 1 0 1 0 1 1 0 1 1 0 0 1 1 1 1 0 0 1 1 1 0 0 1 0 0 0 0 1 1 0 0 1 1 0 0 1 0 1 1 0 1 0 0 1 1 0 1 1 0 1 1 0 0 1 1 0 1 1 1 0 0 0 0 0 0 1 0 \\ 
  0 0 0 1 0 0 1 0 0 0 1 0 0 1 0 0 1 0 1 1 0 0 1 0 0 0 1 1 0 1 1 1 1 0 1 0 1 0 1 0 0 0 1 1 1 0 0 0 1 0 0 0 1 0 1 1 1 0 1 0 1 0 1 1 0 1 1 0 0 0 0 0 0 0 1 
  \end{smallmatrix}\right)\!\!,
  $$
  $$
  \left(\begin{smallmatrix}
  1 1 1 1 1 1 1 1 1 1 1 1 1 1 1 1 1 1 1 1 1 1 1 1 1 1 1 1 1 1 1 0 0 0 0 0 0 0 0 0 0 0 0 0 0 0 0 0 0 0 0 0 0 0 0 0 0 0 0 0 0 0 0 0 0 0 0 1 0 0 0 0 0 0 0 \\ 
  0 0 0 0 0 0 0 0 0 0 0 0 0 0 0 1 1 1 1 1 1 1 1 1 1 1 1 1 1 1 1 1 1 1 1 1 1 1 1 1 1 1 1 1 1 1 0 0 0 0 0 0 0 0 0 0 0 0 0 0 0 0 0 0 0 0 0 0 1 0 0 0 0 0 0 \\ 
  0 0 0 0 0 0 0 1 1 1 1 1 1 1 1 0 0 0 0 0 0 0 0 1 1 1 1 1 1 1 1 0 0 0 0 0 0 0 1 1 1 1 1 1 1 1 1 1 1 1 1 1 1 0 0 0 0 0 0 0 0 0 0 0 0 0 0 0 0 1 0 0 0 0 0 \\ 
  0 0 0 1 1 1 1 0 0 0 0 1 1 1 1 0 0 0 0 1 1 1 1 0 0 0 0 1 1 1 1 0 0 0 1 1 1 1 0 0 0 0 1 1 1 1 0 0 0 1 1 1 1 1 1 1 0 0 0 0 0 0 0 0 0 0 0 0 0 0 1 0 0 0 0 \\ 
  0 1 1 0 0 0 0 0 0 1 1 0 0 1 1 0 0 1 1 0 0 1 1 0 0 1 1 0 0 0 0 0 1 1 0 0 1 1 0 0 0 0 0 0 1 1 0 1 1 0 0 1 1 0 1 1 1 1 1 1 1 1 1 0 0 0 0 0 0 0 0 1 0 0 0 \\ 
  1 0 1 0 0 0 0 0 1 0 1 0 1 0 1 0 1 0 1 0 1 0 1 0 1 0 1 0 0 0 0 1 0 1 0 1 0 1 0 0 0 0 0 1 0 1 1 0 1 0 1 0 1 1 0 1 0 0 0 1 1 1 1 1 1 1 0 0 0 0 0 0 1 0 0 \\ 
  0 1 1 0 0 1 1 0 0 1 1 1 0 1 0 1 0 0 1 1 0 0 1 0 1 0 1 0 0 1 1 0 0 0 1 0 1 0 0 0 1 1 0 0 0 0 0 1 1 0 0 0 0 1 1 0 0 1 1 0 0 1 1 0 1 1 1 0 0 0 0 0 0 1 0 \\ 
  1 0 1 0 1 0 1 0 0 0 0 0 1 0 1 0 1 1 0 0 0 0 0 0 1 1 0 0 1 0 1 0 1 1 0 0 1 1 0 1 0 1 0 0 0 0 1 1 0 0 1 0 1 0 1 1 1 0 1 0 1 0 1 1 0 1 1 0 0 0 0 0 0 0 1 
  \end{smallmatrix}\right)\!\!,
  $$
  or  
  $$
  \left(\begin{smallmatrix}
  1 1 1 1 1 1 1 1 1 1 1 1 1 1 1 1 1 1 1 1 1 1 1 1 1 1 1 1 1 1 1 0 0 0 0 0 0 0 0 0 0 0 0 0 0 0 0 0 0 0 0 0 0 0 0 0 0 0 0 0 0 0 0 0 0 0 0 1 0 0 0 0 0 0 0 \\ 
  0 0 0 0 0 0 0 0 0 0 0 0 0 0 0 1 1 1 1 1 1 1 1 1 1 1 1 1 1 1 1 1 1 1 1 1 1 1 1 1 1 1 1 1 1 1 0 0 0 0 0 0 0 0 0 0 0 0 0 0 0 0 0 0 0 0 0 0 1 0 0 0 0 0 0 \\ 
  0 0 0 0 0 0 0 1 1 1 1 1 1 1 1 0 0 0 0 0 0 0 0 1 1 1 1 1 1 1 1 0 0 0 0 0 0 0 1 1 1 1 1 1 1 1 1 1 1 1 1 1 1 0 0 0 0 0 0 0 0 0 0 0 0 0 0 0 0 1 0 0 0 0 0 \\ 
  0 0 0 1 1 1 1 0 0 0 0 1 1 1 1 0 0 0 0 1 1 1 1 0 0 0 0 1 1 1 1 0 0 0 1 1 1 1 0 0 0 0 1 1 1 1 0 0 0 1 1 1 1 1 1 1 0 0 0 0 0 0 0 0 0 0 0 0 0 0 1 0 0 0 0 \\ 
  0 1 1 0 0 0 0 0 0 1 1 0 0 1 1 0 0 1 1 0 0 1 1 0 0 1 1 0 0 0 0 0 1 1 0 0 1 1 0 0 0 0 0 0 1 1 0 1 1 0 0 1 1 0 1 1 1 1 1 1 1 1 1 0 0 0 0 0 0 0 0 1 0 0 0 \\ 
  1 0 1 0 0 0 0 0 1 0 1 0 1 0 1 0 1 0 1 0 1 0 1 0 1 0 1 0 0 0 0 1 0 1 0 1 0 1 0 0 0 0 0 1 0 1 1 0 1 0 1 0 1 1 0 1 0 0 0 1 1 1 1 1 1 1 0 0 0 0 0 0 1 0 0 \\ 
  0 1 1 0 0 1 1 0 0 1 1 1 0 1 0 1 0 0 1 1 0 0 1 0 1 0 1 0 0 1 1 0 0 0 1 0 1 0 0 0 1 1 0 0 0 0 0 1 1 0 0 0 0 1 1 0 0 1 1 0 0 1 1 0 1 1 1 0 0 0 0 0 0 1 0 \\ 
  1 1 0 0 1 0 1 0 1 1 0 1 1 0 0 1 0 1 0 1 0 1 0 1 1 0 0 0 1 0 1 0 0 0 0 0 1 1 0 1 0 1 1 1 1 1 1 1 0 1 1 1 1 1 0 1 1 0 1 0 1 0 1 1 0 1 1 0 0 0 0 0 0 0 1
  \end{smallmatrix}\right)\!\!.
  $$ 
\end{lemma}
\begin{proof}
  Observe that the set $\cH$ of points that are not covered by the $12$ solids of $\cP$ forms 
  an $8$-divisible set of $75$ points. Let $k$ be the dimension of the span of $\cH$. For the ease of notation we assume that $\cH$ is embedded in $\PG(k-1,2)$ and denote by 
  $a_i$ the number of hyperplanes containing exactly $i$ points from $\cH$. Similar as n the roof of Lemma~\ref{lemma_four_solids_parameters} we use the so standard equations 
  $\sum_{i=0}^8 a_{3+8i}= 2^k-1$, $\sum_{i=0}^8 (3+8i)\cdot a_{3+8i}= 75\cdot\left(2^{k-1}-1\right)$, and $\sum_{i=0}^8 (3+8i)(2+8i)\cdot a_{3+8i}= 75\cdot 74\cdot\left(2^{k-2}-1\right)$ 
  to conclude
  \begin{eqnarray*}
    \sum_{i=0}^8 i \cdot a_{3+8i} &=& 69\cdot 2^{k-4}-9, \\ 
    \sum_{i=0}^8 i^2 \cdot a_{3+8i} &=& 1209\cdot 2^{k-6}-81,\text{ and}\\
    \sum_{i=0}^8 (i-4)(i-5)a_{3+8i} &=& 4\cdot 2^{k-6}-20.
  \end{eqnarray*}  
  Since $(i-4)(i-5)\ge 0$ and $a_{3+8i}\ge 0$ for all $i$, we have $k\ge 8$. Due to the ambient space $\PG(7,2)$ for $\cP$ we are only interested in the case $k=8$ 
  where $\sum_{i=0}^8 (i-4)(i-5)a_{3+8i}=0$. Thus, the unique solution is given by $a_{43}=75$, $a_{35}=180$, and $a_i=0$ otherwise.
  
  Using the notation $[n,k]_2$-code for a binary $k$-dimensional code of effective length $n$, $\cH$ is given by the columns of a projective $[75,8]_2$-two-weight code 
  $C$ with weights $32$ and $40$. Since $\cH$ contains a plane $C$ can be obtained recursively by lengthening $[74,7]_2$, $[72,6]_2$, and $[68,5]_2$ codes with maximum 
  possible column multiplicities at most $2$, $4$, and $8$, respectively, where all non-zero weights are contained in $\{32,40\}$. Using \texttt{LinCode}, see 
  \cite{bouyukliev2021computer}, we enumerated $38$ $[68,5]_2$, $4286$ $[72,6]_2$, $245736$ $[74,7]_2$, and $9964$ $[75,8]_2$ codes. The complete enumeration took 
  85 hours of computation time. Checking which of those $9964$ point sets allow eight disjoint planes leaves just seven possibilities. Using an ILP formulation 
  for a partition of the complement into $12$ solids leaves just the three mentioned cases.      
\end{proof}

\begin{lemma}
   \label{lemma_five_solids_parameters}
  In $\PG(7,2)$ no vector space partition of type $4^{12} 3^8 2^{5-i} 1^{4+3i}$ for $2\le i\le 5$ exists.
\end{lemma}
\begin{proof}
  Assume that $\cP$ is a vector space partition of one of those types. Replacing the lines by their contained three points gives a 
  vector space partition $\cP'$ of type $4^{12} 3^8  1^{19}$, so that we can apply Lemma~\ref{lemma_five_solids_parameters_aux}. 
  Let $\cH$ be the corresponding point set of the points in the subspaces of dimension at most $3$. Using an integer linear programming 
  formulation we have checked that we can pack at most one line if we also pack eight disjoint planes into $\cH$.
\end{proof}
We remark that from the three $4$-divisible projective binary linear codes of length $19$, see e.g.\ \cite{ubt_eref40887}, one contains five disjoint lines and
the other two contain no pair of disjoint lines.

\begin{lemma}
  \label{lemma_4_div_card_20_five_lines}
  Let $\cS$ be a $4$-divisible multiset of points of cardinality $20$ and dimension at most $8$ that contains five disjoint lines, then, up to symmetry, 
  $\cS$ is given by the columns of 
  $$
    \begin{pmatrix}
      1 0 0 0 0 0 1 1 1 0 1 0 0 0 0 1 0 1 1 0\\
      0 1 0 0 0 0 0 0 0 1 1 1 0 1 1 0 1 1 0 0\\
      0 0 1 0 0 0 1 1 1 1 0 1 1 0 0 0 0 0 1 0\\
      0 0 0 1 0 0 1 0 0 1 1 0 1 0 0 0 1 1 1 0\\
      0 0 0 0 1 0 0 1 1 0 0 0 0 1 1 1 1 1 0 0\\
      0 0 0 0 0 1 0 1 1 1 1 0 0 1 0 0 1 0 1 0\\
      0 0 0 0 0 0 0 0 1 1 0 0 0 0 0 0 0 1 0 1
    \end{pmatrix}.
  $$ 
  Moreover, we have $\dim(\cS)=7$ and the spectrum is given by $\left(a_8,a_{12},a_{16}\right)=(67,59,1)$.
\end{lemma}
\begin{proof}
  The projective $4$-divisible binary linear codes of cardinality $20$ have been classified in \cite{ubt_eref40887}. Their counts per dimension are given by $7^2 8^4 9^1$. By a direct enumeration 
  we have checked which corresponding point sets contain five disjoint lines.
\end{proof}

\begin{lemma}
  \label{lemma_six_solids_parameters_1}
  In $\PG(7,2)$ no vector space partition of type $4^{4} 3^{25} 2^{5} 1^{5}$  exists.
\end{lemma}
\begin{proof}
  Assume that $\cP$ is a vector space partition of this type and observe that the set $\cH$ of points that are not covered by the $4$ solids and the $25$ planes of $\cP$ forms 
  an $4$-divisible set of $20$ points that contains five disjoint lines. The unique possibility up to symmetry is determined in Lemma~\ref{lemma_4_div_card_20_five_lines}. 
  It turns out that the ILP formulation for a vector space partition of type $4^{4} 3^{25} 2^{5} 1^{5}$ is infeasible when prescribing these $20$ points.  
\end{proof}

\begin{lemma}
   \label{lemma_six_solids_parameters_2}
   In $\PG(7,2)$ no vector space partition of type $4^{11} 3^{10} 2^{5} 1^{5}$ exists.
\end{lemma}
\begin{proof}
  Assume that $\cP$ is a vector space partition of type $4^{11} 3^{10} 2^{5} 1^{5}$ and observe that the set $\cH$ of points that are not covered by the $11$ solids of $\cP$ forms 
  an $8$-divisible set of $90$ points. Let $k$ be the dimension of the span of $\cH$. For the ease of notation we assume that $\cH$ is embedded in $\PG(k-1,2)$ and denote by 
  $a_i$ the number of hyperplanes containing exactly $i$ points from $\cH$. Similar as n the roof of Lemma~\ref{lemma_four_solids_parameters} we use the standard equations 
  $\sum_{i=0}^{11} a_{2+8i}= 2^k-1$, $\sum_{i=0}^{11} (2+8i)\cdot a_{2+8i}= 90\cdot\left(2^{k-1}-1\right)$, and $\sum_{i=0}^{11} (2+8i)(1+8i)\cdot a_{2+8i}= 90\cdot 89\cdot\left(2^{k-2}-1\right)$ 
  to conclude
  \begin{eqnarray*}
    \sum_{i=0}^{11} i \cdot a_{2+8i} &=& 43\cdot 2^{k-3}-11, \\ 
    \sum_{i=0}^{11} i^2 \cdot a_{2+8i} &=& 3743\cdot 2^{k-7}-121,\text{ and}\\
    \sum_{i=0}^{11} (i-5)(i-6)a_{3+8i} &=& 15\cdot 2^{k-7}-30.
  \end{eqnarray*}  
  Since $(i-5)(i-6)\ge 0$ and $a_{3+8i}\ge 0$ for all $i$, we have $k\ge 8$. Due to the ambient space $\PG(7,2)$ for $\cP$ we are only interested in the case $k=8$ 
  where $\sum_{i=0}^{11} (i-5)(i-6)a_{3+8i}=0$. Thus, the unique solution is given by $a_{50}=90$, $a_{42}=165$, and $a_i=0$ otherwise.

  If $\cH$ contains a solid $S$ in its support, then there are $75$ points in $\cH$ that are not contained in $S$ and each hyperplane contain either $35$ or $43$ points of these. 
  Via projective $[75,8]_2$-two-weight codes $C$ with weights $32$ and $40$ such point sets have already been enumerated in Lemma~\ref{lemma_five_solids_parameters}. Via ILP computations 
  we have filtered out which of the corresponding point sets can be completed by $12$ solids to a vector space partition of $\PG(7,2)$.  
  Only 42~codes remain and we extended them in all possible ways by a four-dimensional simplex code such that the code remains projective. After filtering out isomorphic codes 
  we have again used an ILP formulation to check which point sets can be completed by $11$ solids to a vector space partition of $\PG(7,2)$. 
  For the remaining 245~point sets we have checked which allow to pack $10$ disjoint planes into them via ILP computations. This was possible in 10~cases only and we finally have checked 
  using ILP computations that we can pack at most 4 disjoint lines when we also pack 10 disjoint planes into the point set. In other words, the 20 points that are not covered by the 11 
  solids and the 10 planes do not form a $4$-divisible point set as specified in Lemma~\ref{lemma_4_div_card_20_five_lines}.

  These computations show that in the remaining part we can assume that $\cH$ does not contain a full solid in its support. Now we have enumerated the possibilities of four 
  pairwise disjoint solids in $\PG(7,2)$ up to symmetry. By exhaustive enumeration we have extended the four prescribed solids to 11 solids in total that are pairwise 
  disjoint and discarded all cases that allow the addition of a 12th such solid. For all of these cases we have exhaustively enumerated all vector space 
  partitions of type $4^{11} 3^{10} 1^{20}$ and determined the maximum number of disjoint lines that we can pack into the remaining 20 points. In all cases the answer was 
  at most $4$.    
\end{proof}

We remark that we have also tried to enumerate the projective $[90,8]_2$-two-weight codes $C$ with weights $40$ and $48$ directly. However, we stopped the 
computations after having reached more than $1.5$~million different codes.  

\bigskip

For every type satisfying the numerical conditions of Equation~(\ref{eq_packing_condition}) and Equation~(\ref{eq_dimension_condition}) that is not excluded by one of the 
previous lemmas there indeed exists a corresponding vector space partition of $\PG(7,2)$. We summarize the set of feasible parameters as follows:
\begin{theorem}
  \label{main_thm}
  Let $\cP$ be a vector space partition of $\PG(7,2)$, then $\cP$ has one of the following types:
  \begin{itemize}
    \item $7^2 1^{128}$; 
    \item $6^2 2^{64-i} 1^{3i}$, where $0\le i\le 64$; 
    \item $5^1 3^{32}$; 
    \item $5^1 3^{31-3j} 2^{1-i+7j} 1^{4+3i}$, where $0\le i\le 1+7j$ and $0\le j\le 10$; 
    \item $5^1 3^{29-3j} 2^{7-i+7j} 1^{3i}$, where $0\le i\le 7+7j$ and $0\le j\le 9$; 
    \item $5^1 3^{27-3j} 2^{10-i+7j} 1^{5+3i}$, where $0\le i\le 10+7j$ and $0\le j\le 9$; 
    \item $4^{17}$; 
    \item $4^{16} 3^1 1^8$; 
    \item $4^{16} 2^{5-i} 1^{3i}$, where $0\le i\le 5$; 
    \item $4^{15} 3^2 1^{16}$; 
    \item $4^{15} 3^1 2^{5-i} 1^{8+3i}$, where $0\le i\le 5$; 
    \item $4^{15} 2^{10-i} 1^{3i}$, where $0\le i\le 10$; 
    \item $4^{14} 3^{3-3j} 2^{8-i+7j} 1^{3i}$, where $0\le i\le 8+7j$ and $0\le j\le 1$; 
    \item $4^{14} 3^2 2^{9-i} 1^{4+3i}$, where $0\le i\le 9$; 
    \item $4^{14} 3^1 2^{10-i} 1^{8+3i}$, where $0\le i\le 10$; 
    \item $4^{13} 3^4 2^{8-i} 1^{8+3i}$, where $0\le i\le 8$; 
    \item $4^{13} 3^{3-3j} 2^{13-i+7j} 1^{3i}$, where $0\le i\le 13+7j$ and $0\le j\le 1$; 
    \item $4^{13} 3^2 2^{14-i} 1^{4+3i}$, where $0\le i\le 14$; 
    \item $4^{13} 3^1 2^{16-i} 1^{5+3i}$where $0\le i\le 16$; 
    \item $4^{12} 3^8 2^{1-i} 1^{16+3i}$, where $0\le i\le 1$; 
    \item $4^{12} 3^{7-3j} 2^{7-i+7j} 1^{5+3i}$,  where $0\le i\le 7+7j$ and $0\le j\le 2$; 
    \item $4^{12} 3^{6-3j} 2^{11-i+7j} 1^{3i}$,  where $0\le i\le 11+7j$ and $0\le j\le 2$; 
    \item $4^{12} 3^{5-3j} 2^{12-i+7j} 1^{4+3i}$,  where $0\le i\le 11+7j$ and $0\le j\le 1$; 
    \item $4^{11} 3^10 2^{4-i} 1^{8+3i}$, where $0\le i\le 4$; 
    \item $4^{11} 3^{9-3j} 2^{9-i+7j} 1^{3i}$, where $0\le i\le 9+7j$ and $0\le j\le 3$; 
    \item $4^{11} 3^{8-3j} 2^{10-i+7j} 1^{4+3i}$, where $0\le i\le 10+7j$ and $0\le j\le 2$; 
    \item $4^{11} 3^{7-3j} 2^{12-i+7j} 1^{5+3i}$, where $0\le i\le 12+7j$ and $0\le j\le 2$; 
    \item $4^{10} 3^{15}$; 
    \item $4^{10} 3^{14-3j} 2^{1-i+7j} 1^{4+3i}$, where $0\le i\le 1+7j$ and $0\le j\le 4$; 
    \item $4^{10} 3^{13} 2^{2-i} 1^{8+3i}$, where $0\le i\le 2$; 
    \item $4^{10} 3^{12-3j} 2^{7-i+7j} 1^{3i}$, where $0\le i\le 7+7j$ and $0\le j\le 4$; 
    \item $4^{10} 3^{10-3j} 2^{10-i+7j} 1^{5+3i}$, where $0\le i\le 10+7j$ and $0\le j\le 3$; 
    \item $4^9 3^{16} 1^8$; 
    \item $4^9 3^{15-3j} 2^{5-i+7j} 1^{3i}$, where $0\le i\le 5+7j$ and $0\le j\le 5$; 
    \item $4^9 3^{14-3j} 2^{6-i+7j} 1^{4+3i}$, where $0\le i\le 6+7j$ and $0\le j\le 4$; 
    \item $4^9 3^{13-3j} 2^{8-i+7j} 1^{5+3i}$, where $0\le i\le 8+7j$ and $0\le j\le 4$; 
    \item $4^8 3^{17-3j} 2^{4-i+7j} 1^{4+3i}$, where $0\le i\le 4+7j$ and $0\le j\le 5$; 
    \item $4^8 3^{16-3j} 2^{6-i+7j} 1^{5+3i}$, where $0\le i\le 6+7j$ and $0\le j\le 5$; 
    \item $4^8 3^{15-3j} 2^{10-i+7j} 1^{3i}$, where $0\le i\le 10+7j$ and $0\le j\le 5$; 
    \item $4^7 3^{19} 2^{3-i} 1^{8+3i}$, where $0\le i\le 3$; 
    \item $4^7 3^{18-3j} 2^{8-i+7j} 1^{3i}$, where $0\le i\le 8+7j$ and $0\le j\le 6$; 
    \item $4^7 3^{17-3j} 2^{9-i+7j} 1^{4+3i}$, where $0\le i\le 9+7j$ and $0\le j\le 5$; 
    \item $4^7 3^{16-3j} 2^{11-i+7j} 1^{5+3i}$, where $0\le i\le 11+7j$ and $0\le j\le 5$; 
    \item $4^6 3^{21-3j} 2^{6-i+7j} 1^{3i}$, where $0\le i\le 6+7j$ and $0\le j\le 7$; 
    \item $4^6 3^{20-3j} 2^{7-i+7j} 1^{4+3i}$, where $0\le i\le 7+7j$ and $0\le j\le 6$; 
    \item $4^6 3^{19-3j} 2^{9-i+7j} 1^{5+3i}$, where $0\le i\le 9+7j$ and $0\le j\le 6$; 
    \item $4^5 3^{23-3j} 2^{5-i+7j} 1^{4+3i}$, where $0\le i\le 5+7j$ and $0\le j\le 7$; 
    \item $4^5 3^{22-3j} 2^{7-i+7j} 1^{5+3i}$, where $0\le i\le 7+7j$ and $0\le j\le 7$; 
    \item $4^5 3^{21-3j} 2^{11-i+7j} 1^{3i}$, where $0\le i\le 11+7j$ and $0\le j\le 7$; 
    \item $4^4 3^{25} 2^{4-i} 1^{8+3i}$, where $0\le i\le 4$; 
    \item $4^4 3^{24-3j} 2^{9-i+7j} 1^{3i}$, where $0\le i\le 9+7j$ and $0\le j\le 8$; 
    \item $4^4 3^{23-3j} 2^{10-i+7j} 1^{4+3i}$, where $0\le i\le 10+7j$ and $0\le j\le 7$; 
    \item $4^4 3^{22-3j} 2^{12-i+7j} 1^{5+3i}$, where $0\le i\le 12+7j$ and $0\le j\le 7$; 
    \item $4^3 3^{29}  2^{1-i} 1^4+3i$, where $0\le i\le 1$; 
    \item $4^3 3^{28} 2^{2-i} 1^{8+3i}$, where $0\le i\le 2$; 
    \item $4^3 3^{30-3j} 2^{7j-i} 1^{3i}$, where $0\le i\le 7j$ and $0\le j\le 10$; 
    \item $4^3 3^{26-3j} 2^{8-i+7j} 1^{4+3i}$, where $0\le i\le 8+7j$ and $0\le j\le 8$; 
    \item $4^3 3^{25-3j} 2^{10-i+7j} 1^{5+3i}$, where $0\le i\le 10+7j$ and $0\le j\le 8$; 
    \item $4^2 3^{31}  1^8$; 
    \item $4^2 3^{30-3j} 2^{5-i+7j} 1^{3i}$, where $0\le i\le 5+7j$ and $0\le j\le 10$; 
    \item $4^2 3^{29-3j} 2^{6-i+7j} 1^{4+3i}$, where $0\le i\le 6+7j$ and $0\le j\le 9$; 
    \item $4^2 3^{28-3j} 2^{8-i+7j} 1^{5+3i}$, where $0\le i\le 8+7j$ and $0\le j\le 9$; 
    \item $4^1 3^{32-3j} 2^{4-i+7j} 1^{4+3i}$, where $0\le i\le 4+7j$ and $0\le j\le 10$; 
    \item $4^1 3^{31-3j} 2^{6-i+7j} 1^{5+3i}$, where $0\le i\le 6+7j$ and $0\le j\le 10$; 
    \item $4^1 3^{30-3j} 2^{10-i+7j} 1^{3i}$, where $0\le i\le 10+7j$ and $0\le j\le 10$; 
    \item $4^0 3^{34} 2^{3-i} 1^{8+3i}$, where $0\le i\le 3$; 
    \item $4^0 3^{33-3j} 2^{8-i+7j} 1^{3i}$, where $0\le i\le 8+7j$ and $0\le j\le 11$; 
    \item $4^0 3^{32-3j} 2^{9-i+7j} 1^{4+3i}$, where $0\le i\le 9+7j$ and $0\le j\le 10$; 
    \item $4^0 3^{31-3j} 2^{11-i+7j} 1^{5+3i}$, where $0\le i\le 11+7j$ and $0\le j\le 10$; 
  \end{itemize}
\end{theorem}

In Proposition~\ref{prop_vsp_pg_7_2_ext} in the appendix we also give a more explicit and extensive variant of the list of feasible types of vector space partitions of $\PG(7,2)$. 
A more compact variant is stated in Theorem~\ref{main_thm_alternative}.

\section{Vector space partitions in $\PG(v-1,q)$ for $v\le 5$}
\label{sec_vsp_pg_small} 
The aim of this short section is to discuss the possible types of vector space partitions of $\PG(v-1,q)$ for all dimensions $v\le 5$ and arbitrary field size $q$. For
$v=1$ the ambient space consists of a single point itself, so that no vector space partition exists since we assume a maximum dimension of $v-1$ for the elements of a vector 
space partition. For the same reason all elements of a vector space partition of $\PG(1,q)$ have dimension $1$. I.e., the only possible type is $1^{m_1}$ where $m_1=[2]_q=q+1$. 
For $v=3$ the dimension condition yields that all elements of a vector space partition $\cP$ of $\PG(2,q)$ have either dimension $1$ or $2$ and that dimension $2$ can occur at most 
once. If $\cP$ does not contain an element of dimension $2$, then its type is given by $1^{m_1}$ where $m_1=[3]_q=q^2+q+1$. In that case $\cP$ is reducible since we may choose any 
line and choose it as an element of the vector space partition instead of its $q+1$ contained points. If $\cP$ is of type $2^1 1^{m_1}$, then we have $m_1=[3]_q-[2]_q=q^2$.

\begin{proposition}
  The possible types of vector space partitions of $\PG(3,q)$ are given by $2^{q^2+1-j} 1^{(q+1)j}$, where $0\le j\le q^2+1$, and $3^1 1^{q^3}$.
\end{proposition}
\begin{proof}
  Directly implied by the packing and the dimension condition in equations~(\ref{eq_packing_condition}) and (\ref{eq_dimension_condition}).
\end{proof}
All of the mentioned types can indeed be attained. A Desarguesian line spread in $\PG(3,q)$ yields a vector space partition of type $2^{q^2+1}$ where we may replace 
arbitrary $j$ lines by their contained points. Choosing an arbitrary plane in the ambient space leaves $q^3$ uncovered points. The latter vector space partition is
irreducible. For vector space partitions of type $2^{q^2+1-j} 1^{(q+1)j}$ it is an interesting question which values of $j$ do allow an irreducible vector 
space partition of that type. This problem is equivalent to the classification of the possible sizes of (inclusion) maximal partial line spreads in $\PG(3,q)$, 
see e.g.\ \cite{gacs2003maximal}.

For vector space partitions of type $2^{m_2}1^{m_1}$ of $\PG(4,q)$ the packing condition in Equation~(\ref{eq_packing_condition}) 
only implies $m_2=q^3+q-j$ and $m_1=1+j(q+1)$ for $0\le j\le q^3+q$. However, the $1+j(q+1)$ points that are not covered by the lines have to correspond to a (projective) $q$-divisible 
linear codes over $\mathbb{F}_q$ of effective length $1+j(q+1)$. Using the characterization result for the possible length of $q^r$-divisible codes over $\mathbb{F}_q$ from
\cite{kiermaier2020lengths} we can conclude $j\ge q-1$. Using this and the packing and the dimension condition in equations~(\ref{eq_packing_condition}) and (\ref{eq_dimension_condition}), 
we conclude: 
\begin{proposition}
  The possible types of vector space partitions of $\PG(4,q)$ are given by $4^1 1^{q^4}$, $3^1 2^{q^3-j} 1^{j(q+1)}$ for $0\le j\le q^3$, and $2^{q^3+1-j}1^{q^2+j(q+1)}$ for $0\le j\le q^3+1$.
\end{proposition}
All of the mentioned types can indeed be attained. Choosing an arbitrary solid in the ambient space leaves $q^4$ uncovered points. A lifted MRD code gives rise to a vector space 
partition $\cP$ of type $3^1 2^{q^3}$. In $\cP$ we can either replace $j$ lines by their contained points or replace the plane by a line and $q^2$ points to obtain a vector 
space partition $\cP'$ of type $2^{q^3+1} 1^{q^2}$. In $\cP'$ we can then replace replace $j$ lines by their contained points.

\medskip

Using the same methods one can easily characterize all feasible types of vector space partitions in $\PG(5,q)$ that do not contain a plane:
\begin{itemize}
  \item $5^1 1^{q^5}$;
  \item $4^1 2^{q^4-j} 1^{j(q+1)}$ for $0\le j\le q^4$;
  \item $2^{q^4+q^2+1-j} 1^{j(q+1)}$ for $0\le j\le q^4+q^2+1$.
\end{itemize}
A plane spread in $\PG(5,q)$ is a vector space partition of type $3^{q^3+1}$. From there we can easily obtain vector space partitions of type 
$3^{q^3+1-j} 2^{j-i} 1^{i(q+1)}$ for all $0\le j\le q^3+1$ and all $0\le i\le j$.  However, also vector space partitions of other types do exist. 
With increasing dimension of the ambient space the problem of the classification of all feasible types of vector space partitions gets harder and 
harder. We would like to mention that in $\PG(7,3)$ the maximum number $A_3(8,6;3)$ of pairwise disjoint planes is unknown. The currently best known 
bounds are $244\le A_3(8,6;3) \le 248$, see e.g.\ \cite{honold2018partial}. If $248$ such pairwise disjoint planes exist in $\PG(7,3)$, then the $56$ 
uncovered points have to form a so-called \emph{Hill cap} \cite{hill1978caps} corresponding to a two-weight code. Since the support of this object does 
not contain a line, there is e.g.\ no vector space partition of type $3^{248} 2^1 1^{52}$ in $\PG(7,3)$. 



\newcommand{\etalchar}[1]{$^{#1}$}

sascha@sascha-ThinkPad

\appendix
\section{A more explicit variant of the main theorem}
During the paper we have concluded several forbidden supertails. For the ease of the reader we summarize those that have been necessary in the classification of 
the feasible types of vector space partitions of $\PG(7,2)$:
\begin{proposition}
  \label{prop_forbidden_supertails}
  No vector space partition of $\PG(v-1,2)$ exists if it has a supertail of one of the following types:
  \begin{itemize}
    \item $1^i$ for $1\le i\le 2$;
    \item $2^0 1^4$;
    \item $2^{2-i} 1^{2+3i}$ for $0\le i\le 1$;
    \item $2^{3-i} 1^{3i}$ for all $0\le i\le 3$;
    \item $2^{4-i} 1^{3i}$ for all $0\le i\le 4$;
    \item $2^{4-i} 1^{1+3i}$ for all $0\le i\le 4$;
    \item $2^3 1^5$; 
    \item $2^4 1^5$;
    \item $3^3 2^{3-i} 1^{3i}$ for $0\le i\le 3$;
    \item $3^2 2^{5-i} 1^{1+3i}$ for all $0\le i\le 4$;
    \item $3^1 2^6 1^5$;
    \item $3^1 2^{11} 1^5$.
  \end{itemize}
\end{proposition}  
With this we can reformulate Theorem~\ref{main_thm} as follows:
\begin{theorem}
  \label{main_thm_alternative}
  Let $\cP$ be a vector space partition of $\PG(7,2)$ of type $7^{m_7}\dot 1^{m_1}$ satisfying the packing condition in Equation~(\ref{eq_packing_condition}) and 
  the dimension condition in Equation~(\ref{eq_dimension_condition}) as well as the special tail conditions from Proposition~\ref{prop_forbidden_supertails}. 
  Then, the type of $\cP$ is not contained in the following exhaustive list:
  \begin{itemize}
    \item $4^{13} 3^6 2^{6-i} 1^{3i}$ for $0\le i\le 6$; 
    \item $4^{13} 3^5 2^{7-i} 1^{4+3i}$ for $0 \le i \le 7$;
    \item $4^{13} 3^4 2^9 1^5$;
    \item $4^{12} 3^8 2^{5-i} 1^{4+3i}$ for $2 \le i \le 5$;
    \item $4^{11} 3^{10} 2^5 15$;
    \item $4^4 3^{25} 2^5 1^5$.
  \end{itemize}
\end{theorem}

A more explicit and extensive variant of our main theorem is given as follows:
\begin{proposition}
  \label{prop_vsp_pg_7_2_ext}
  Let $\cP$ be a vector space partition of $\PG(7,2)$, then $\cP$ has one of the following types:
  \begin{itemize}
    \item $7^2 1^{128}$; 
    \item $6^2 2^{64-i} 1^{3i}$, where $0\le i\le 64$; 
    \item $5^1 3^{32}$; 
    \item $5^1 3^{31} 2^{1-i} 1^{4+3i}$, where $0\le i\le 1$; 
    \item $5^1 3^{29} 2^{7-i} 1^{3i}$, where $0\le i\le 7$; 
    \item $5^1 3^{28} 2^{8-i} 1^{4+3i}$, where $0\le i\le 8$; 
    \item $5^1 3^{27} 2^{10-i} 1^{5+3i}$, where $0\le i\le 10$; 
    \item $5^1 3^{26} 2^{14-i} 1^{3i}$, where $0\le i\le 14$; 
    \item $5^1 3^{25} 2^{15-i} 1^{4+3i}$, where $0\le i\le 15$; 
    \item $5^1 3^{24} 2^{17-i} 1^{5+3i}$, where $0\le i\le 17$; 
    \item $5^1 3^{23} 2^{21-i} 1^{3i}$, where $0\le i\le 21$; 
    \item $5^1 3^{22} 2^{22-i} 1^{4+3i}$, where $0\le i\le 22$; 
    \item $5^1 3^{21} 2^{24-i} 1^{5+3i}$, where $0\le i\le 24$; 
    \item $5^1 3^{20} 2^{28-i} 1^{3i}$, where $0\le i\le 28$; 
    \item $5^1 3^{19} 2^{29-i} 1^{4+3i}$, where $0\le i\le 29$; 
    \item $5^1 3^{18} 2^{31-i} 1^{5+3i}$, where $0\le i\le 31$; 
    \item $5^1 3^{17} 2^{35-i} 1^{3i}$, where $0\le i\le 35$; 
    \item $5^1 3^{16} 2^{36-i} 1^{4+3i}$, where $0\le i\le 36$; 
    \item $5^1 3^{15} 2^{38-i} 1^{5+3i}$, where $0\le i\le 38$; 
    \item $5^1 3^{14} 2^{42-i} 1^{3i}$, where $0\le i\le 42$; 
    \item $5^1 3^{13} 2^{43-i} 1^{4+3i}$, where $0\le i\le 43$; 
    \item $5^1 3^{12} 2^{45-i} 1^{5+3i}$, where $0\le i\le 45$; 
    \item $5^1 3^{11} 2^{49-i} 1^{3i}$, where $0\le i\le 49$; 
    \item $5^1 3^{10} 2^{50-i} 1^{4+3i}$, where $0\le i\le 50$; 
    \item $5^1 3^9 2^{52-i} 1^{5+3i}$, where $0\le i\le 52$; 
    \item $5^1 3^8 2^{56-i} 1^{3i}$, where $0\le i\le 56$; 
    \item $5^1 3^7 2^{57-i} 1^{4+3i}$, where $0\le i\le 57$; 
    \item $5^1 3^6 2^{59-i} 1^{5+3i}$, where $0\le i\le 59$; 
    \item $5^1 3^5 2^{63-i} 1^{3i}$, where $0\le i\le 63$; 
    \item $5^1 3^4 2^{64-i} 1^{4+3i}$, where $0\le i\le 64$; 
    \item $5^1 3^3 2^{66-i} 1^{5+3i}$, where $0\le i\le 66$; 
    \item $5^1 3^2 2^{70-i} 1^{3i}$, where $0\le i\le 70$; 
    \item $5^1 3^1 2^{71-i} 1^{4+3i}$, where $0\le i\le 71$; 
    \item $5^1 2^{73-i} 1^{5+3i}$, where $0\le i\le 73$; 
    \item $4^{17}$; 
    \item $4^{16} 3^1 1^8$; 
    \item $4^{16} 2^{5-i} 1^{3i}$, where $0\le i\le 5$; 
    \item $4^{15} 3^2 1^{16}$; 
    \item $4^{15} 3^1 2^{5-i} 1^{8+3i}$, where $0\le i\le 5$; 
    \item $4^{15} 2^{10-i} 1^{3i}$, where $0\le i\le 10$; 
    \item $4^{14} 3^3 2^{8-i} 1^{3i}$, where $0\le i\le 8$; 
    \item $4^{14} 3^2 2^{9-i} 1^{4+3i}$, where $0\le i\le 9$; 
    \item $4^{14} 3^1 2^{10-i} 1^{8+3i}$, where $0\le i\le 10$; 
    \item $4^{14} 3^0 2^{15-i} 1^{3i}$, where $0\le i\le 15$; 
    \item $4^{13} 3^4 2^{8-i} 1^{8+3i}$, where $0\le i\le 8$; 
    \item $4^{13} 3^3 2^{13-i} 1^{3i}$, where $0\le i\le 13$; 
    \item $4^{13} 3^2 2^{14-i} 1^{4+3i}$, where $0\le i\le 14$; 
    \item $4^{13} 3^1 2^{16-i} 1^{5+3i}$where $0\le i\le 16$; 
    \item $4^{13} 3^0 2^{20-i} 1^{3i}$, where $0\le i\le 20$; 
    \item $4^{12} 3^8 2^{1-i} 1^{16+3i}$, where $0\le i\le 1$; 
    \item $4^{12} 3^7 2^{7-i} 1^{5+3i}$,  where $0\le i\le 7$; 
    \item $4^{12} 3^6 2^{11-i} 1^{3i}$,  where $0\le i\le 11$; 
    \item $4^{12} 3^5 2^{12-i} 1^{4+3i}$,  where $0\le i\le 11$; 
    \item $4^{12} 3^4 2^{14-i} 1^{5+3i}$, where $0\le i\le 14$; 
    \item $4^{12} 3^3 2^{18-i} 1^{3i}$, where $0\le i\le 18$; 
    \item $4^{12} 3^2 2^{19-i} 1^{4+3i}$, where $0\le i\le 19$; 
    \item $4^{12} 3^1 2^{21-i} 1^{5+3i}$, where $0\le i\le 21$; 
    \item $4^{12} 2^{25-i} 1^{3i}$, where $0\le i\le 25$; 
    \item $4^{11} 3^10 2^{4-i} 1^{8+3i}$, where $0\le i\le 4$; 
    \item $4^{11} 3^9 2^{9-i} 1^{3i}$, where $0\le i\le 9$; 
    \item $4^{11} 3^8 2^{10-i} 1^{4+3i}$, where $0\le i\le 10$; 
    \item $4^{11} 3^7 2^{12-i} 1^{5+3i}$, where $0\le i\le 12$; 
    \item $4^{11} 3^6 2^{16-i} 1^0{3i}$, where $0\le i\le 16$; 
    \item $4^{11} 3^5 2^{17-i} 1^{4+3i}$, where $0\le i\le 17$; 
    \item $4^{11} 3^4 2^{19-i} 1^{5+3i}$, where $0\le i\le 19$; 
    \item $4^{11} 3^3 2^{23-i} 1^{3i}$, where $0\le i\le 23$; 
    \item $4^{11} 3^2 2^{24-i} 1^{4+3i}$, where $0\le i\le 24$; 
    \item $4^{11} 3^1 2^{26-i} 1^{5+3i}$, where $0\le i\le 26$; 
    \item $4^{11} 2^{30-i} 1^{3i}$, where $0\le i\le 30$; 
    \item $4^{10} 3^{15}$; 
    \item $4^{10} 3^{14} 2^{1-i} 1^{4+3i}$, where $0\le i\le 1$; 
    \item $4^{10} 3^{13} 2^{2-i} 1^{8+3i}$, where $0\le i\le 2$; 
    \item $4^{10} 3^{12} 2^{7-i} 1^{3i}$, where $0\le i\le 7$; 
    \item $4^{10} 3^{11} 2^{8-i} 1^{4+3i}$, where $0\le i\le 8$; 
    \item $4^{10} 3^{10} 2^{10-i} 1^{5+3i}$, where $0\le i\le 10$; 
    \item $4^{10} 3^9 2^{14-i} 1^{3i}$, where $0\le i\le 14$; 
    \item $4^{10} 3^8 2^{15-i} 1^{4+3i}$, where $0\le i\le 15$; 
    \item $4^{10} 3^7 2^{17-i} 1^{5+3i}$, where $0\le i\le 17$; 
    \item $4^{10} 3^6 2^{21-i} 1^{3i}$, where $0\le i\le 21$; 
    \item $4^{10} 3^5 2^{22-i} 1^{4+3i}$, where $0\le i\le 22$; 
    \item $4^{10} 3^4 2^{24-i} 1^{5+3i}$, where $0\le i\le 24$; 
    \item $4^{10} 3^3 2^{28-i} 1^{3i}$, where $0\le i\le 28$; 
    \item $4^{10} 3^2 2^{29-i} 1^{4+3i}$, where $0\le i\le 29$; 
    \item $4^{10} 3^1 2^{31-i} 1^{5+3i}$, where $0\le i\le 31$; 
    \item $4^{10} 2^{35-i} 1^{3i}$, where $0\le i\le 35$; 
    \item $4^9 3^{16} 1^8$; 
    \item $4^9 3^{15} 2^{5-i} 1^{3i}$, where $0\le i\le 5$; 
    \item $4^9 3^{14} 2^{6-i} 1^{4+3i}$, where $0\le i\le 6$; 
    \item $4^9 3^{13} 2^{8-i} 1^{5+3i}$, where $0\le i\le 8$; 
    \item $4^9 3^{12} 2^{12-i} 1^{3i}$, where $0\le i\le 12$; 
    \item $4^9 3^{11} 2^{13-i} 1^{4+3i}$, where $0\le i\le 13$; 
    \item $4^9 3^{10} 2^{15-i} 1^{5+3i}$, where $0\le i\le 15$; 
    \item $4^9 3^9 2^{19-i} 1^{3i}$, where $0\le i\le 19$; 
    \item $4^9 3^8 2^{20-i} 1^{4+3i}$, where $0\le i\le 20$; 
    \item $4^9 3^7 2^{22-i} 1^{5+3i}$, where $0\le i\le 22$; 
    \item $4^9 3^6 2^{26-i} 1^{3i}$, where $0\le i\le 26$; 
    \item $4^9 3^5 2^{27-i} 1^{4+3i}$, where $0\le i\le 27$; 
    \item $4^9 3^4 2^{29-i} 1^{5+3i}$, where $0\le i\le 29$; 
    \item $4^9 3^3 2^{33-i} 1^{3i}$, where $0\le i\le 33$; 
    \item $4^9 3^2 2^{34-i} 1^{4+3i}$, where $0\le i\le 34$; 
    \item $4^9 3^1 2^{36-i} 1^{5+3i}$, where $0\le i\le 36$; 
    \item $4^9 2^{40-i} 1^{3i}$, where $0\le i\le 40$; 
    \item $4^8 3^{17} 2^{4-i} 1^{4+3i}$, where $0\le i\le 4$; 
    \item $4^8 3^{16} 2^{6-i} 1^{5+3i}$, where $0\le i\le 6$; 
    \item $4^8 3^{15} 2^{10-i} 1^{3i}$, where $0\le i\le 10$; 
    \item $4^8 3^{14} 2^{11-i} 1^{4+3i}$, where $0\le i\le 11$; 
    \item $4^8 3^{13} 2^{13-i} 1^{5+3i}$, where $0\le i\le 13$; 
    \item $4^8 3^{12} 2^{17-i} 1^{3i}$, where $0\le i\le 17$; 
    \item $4^8 3^{11} 2^{18-i} 1^{4+3i}$, where $0\le i\le 18$; 
    \item $4^8 3^{10} 2^{20-i} 1^{5+3i}$, where $0\le i\le 20$; 
    \item $4^8 3^9 2^{24-i} 1^{3i}$, where $0\le i\le 24$; 
    \item $4^8 3^8 2^{25-i} 1^{4+3i}$, where $0\le i\le 25$; 
    \item $4^8 3^7 2^{27-i} 1^{5+3i}$, where $0\le i\le 27$; 
    \item $4^8 3^6 2^{31-i} 1^{3i}$, where $0\le i\le 31$; 
    \item $4^8 3^5 2^{32-i} 1^{4+3i}$, where $0\le i\le 32$; 
    \item $4^8 3^4 2^{34-i} 1^{5+3i}$, where $0\le i\le 34$; 
    \item $4^8 3^3 2^{38-i} 1^{3i}$, where $0\le i\le 38$; 
    \item $4^8 3^2 2^{39-i} 1^{4+3i}$, where $0\le i\le 39$; 
    \item $4^8 3^1 2^{41-i} 1^{5+3i}$, where $0\le i\le 41$; 
    \item $4^8 3^0 2^{45-i} 1^{3i}$, where $0\le i\le 45$; 
    \item $4^7 3^{19} 2^{3-i} 1^{8+3i}$, where $0\le i\le 3$; 
    \item $4^7 3^{18} 2^{8-i} 1^{3i}$, where $0\le i\le 8$; 
    \item $4^7 3^{17} 2^{9-i} 1^{4+3i}$, where $0\le i\le 9$; 
    \item $4^7 3^{16} 2^{11-i} 1^{5+3i}$, where $0\le i\le 11$; 
    \item $4^7 3^{15} 2^{15-i} 1^0{3i}$, where $0\le i\le 15$; 
    \item $4^7 3^{14} 2^{16-i} 1^{4+3i}$, where $0\le i\le 16$; 
    \item $4^7 3^{13} 2^{18-i} 1^{5+3i}$, where $0\le i\le 18$; 
    \item $4^7 3^{12} 2^{22-i} 1^{3i}$, where $0\le i\le 22$; 
    \item $4^7 3^{11} 2^{23-i} 1^{4+3i}$, where $0\le i\le 23$; 
    \item $4^7 3^{10} 2^{25-i} 1^{5+3i}$, where $0\le i\le 25$; 
    \item $4^7 3^9 2^{29-i} 1^{3i}$, where $0\le i\le 29$; 
    \item $4^7 3^8 2^{30-i} 1^{4+3i}$, where $0\le i\le 30$; 
    \item $4^7 3^7 2^{32-i} 1^{5+3i}$, where $0\le i\le 32$; 
    \item $4^7 3^6 2^{36-i} 1^{3i}$, where $0\le i\le 36$; 
    \item $4^7 3^5 2^{37-i} 1^{4+3i}$, where $0\le i\le 37$; 
    \item $4^7 3^4 2^{39-i} 1^{5+3i}$, where $0\le i\le 39$; 
    \item $4^7 3^3 2^{43-i} 1^{3i}$, where $0\le i\le 43$; 
    \item $4^7 3^2 2^{44-i} 1^{4+3i}$, where $0\le i\le 44$; 
    \item $4^7 3^1 2^{46-i} 1^{5+3i}$, where $0\le i\le 46$; 
    \item $4^7  2^{50-i} 1^{3i}$, where $0\le i\le 50$; 
    \item $4^6 3^{21} 2^{6-i} 1^{3i}$, where $0\le i\le 6$; 
    \item $4^6 3^{20} 2^{7-i} 1^{4+3i}$, where $0\le i\le 7$; 
    \item $4^6 3^{19} 2^{9-i} 1^{5+3i}$, where $0\le i\le 9$; 
    \item $4^6 3^{18} 2^{13-i} 1^{3i}$, where $0\le i\le 13$; 
    \item $4^6 3^{17} 2^{14-i} 1^{4+3i}$, where $0\le i\le 14$; 
    \item $4^6 3^{16} 2^{16-i} 1^{5+3i}$, where $0\le i\le 16$; 
    \item $4^6 3^{15} 2^{20-i} 1^{3i}$, where $0\le i\le 20$; 
    \item $4^6 3^{14} 2^{21-i} 1^{4+3i}$, where $0\le i\le 21$; 
    \item $4^6 3^{13} 2^{23-i} 1^{5+3i}$, where $0\le i\le 23$; 
    \item $4^6 3^{12} 2^{27-i} 1^{3i}$, where $0\le i\le 27$; 
    \item $4^6 3^{11} 2^{28-i} 1^{4+3i}$, where $0\le i\le 28$; 
    \item $4^6 3^{10} 2^{30-i} 1^{5+3i}$, where $0\le i\le 30$; 
    \item $4^6 3^9 2^{34-i} 1^{3i}$, where $0\le i\le 34$; 
    \item $4^6 3^8 2^{35-i} 1^{4+3i}$, where $0\le i\le 35$; 
    \item $4^6 3^7 2^{37-i} 1^{5+3i}$, where $0\le i\le 37$; 
    \item $4^6 3^6 2^{41-i} 1^{3i}$, where $0\le i\le 41$; 
    \item $4^6 3^5 2^{42-i} 1^{4+3i}$, where $0\le i\le 42$; 
    \item $4^6 3^4 2^{44-i} 1^{5+3i}$, where $0\le i\le 44$; 
    \item $4^6 3^3 2^{48-i} 1^{3i}$, where $0\le i\le 48$; 
    \item $4^6 3^2 2^{49-i} 1^{4+3i}$, where $0\le i\le 49$; 
    \item $4^6 3^1 2^{51-i} 1^{5+3i}$, where $0\le i\le 51$; 
    \item $4^6 3^0 2^{55-i} 1^{3i}$, where $0\le i\le 55$; 
    \item $4^5 3^{23} 2^{5-i} 1^{4+3i}$, where $0\le i\le 5$; 
    \item $4^5 3^{22} 2^{7-i} 1^{5+3i}$, where $0\le i\le 7$; 
    \item $4^5 3^{21} 2^{11-i} 1^{3i}$, where $0\le i\le 11$; 
    \item $4^5 3^{20} 2^{12-i} 1^{4+3i}$, where $0\le i\le 12$; 
    \item $4^5 3^{19} 2^{14-i} 1^{5+3i}$, where $0\le i\le 14$; 
    \item $4^5 3^{18} 2^{18-i} 1^{3i}$, where $0\le i\le 18$; 
    \item $4^5 3^{17} 2^{19-i} 1^{4+3i}$, where $0\le i\le 19$; 
    \item $4^5 3^{16} 2^{21-i} 1^{5+3i}$, where $0\le i\le 21$; 
    \item $4^5 3^{15} 2^{25-i} 1^{3i}$, where $0\le i\le 25$; 
    \item $4^5 3^{14} 2^{26-i} 1^{4+3i}$, where $0\le i\le 26$; 
    \item $4^5 3^{13} 2^{28-i} 1^{5+3i}$, where $0\le i\le 28$; 
    \item $4^5 3^{12} 2^{32-i} 1^{3i}$, where $0\le i\le 32$; 
    \item $4^5 3^{11} 2^{33-i} 1^{4+3i}$, where $0\le i\le 33$; 
    \item $4^5 3^{10} 2^{35-i} 1^{5+3i}$, where $0\le i\le 35$; 
    \item $4^5 3^9 2^{39-i} 1^{3i}$, where $0\le i\le 39$; 
    \item $4^5 3^8 2^{40-i} 1^{4+3i}$, where $0\le i\le 40$; 
    \item $4^5 3^7 2^{42-i} 1^{5+3i}$, where $0\le i\le 42$; 
    \item $4^5 3^6 2^{46-i} 1^{3i}$, where $0\le i\le 46$; 
    \item $4^5 3^5 2^{47-i} 1^{4+3i}$, where $0\le i\le 47$; 
    \item $4^5 3^4 2^{49-i} 1^{5+3i}$, where $0\le i\le 49$; 
    \item $4^5 3^3 2^{53-i} 1^{3i}$, where $0\le i\le 53$; 
    \item $4^5 3^2 2^{54-i} 1^{4+3i}$, where $0\le i\le 54$; 
    \item $4^5 3^1 2^{56-i} 1^{5+3i}$, where $0\le i\le 56$; 
    \item $4^5 2^{60-i} 1^{3i}$, where $0\le i\le 60$; 
    \item $4^4 3^{25} 2^{4-i} 1^{8+3i}$, where $0\le i\le 4$; 
    \item $4^4 3^{24} 2^{9-i} 1^{3i}$, where $0\le i\le 9$; 
    \item $4^4 3^{23} 2^{10-i} 1^{4+3i}$, where $0\le i\le 10$; 
    \item $4^4 3^{22} 2^{12-i} 1^{5+3i}$, where $0\le i\le 12$; 
    \item $4^4 3^{21} 2^{16-i} 1^{3i}$, where $0\le i\le 16$; 
    \item $4^4 3^{20} 2^{17-i} 1^{4+3i}$, where $0\le i\le 17$; 
    \item $4^4 3^{19} 2^{19-i} 1^{5+3i}$, where $0\le i\le 19$; 
    \item $4^4 3^{18} 2^{23-i} 1^{3i}$, where $0\le i\le 23$; 
    \item $4^4 3^{17} 2^{24-i} 1^{4+3i}$, where $0\le i\le 24$; 
    \item $4^4 3^{16} 2^{26-i} 1^{5+3i}$, where $0\le i\le 26$; 
    \item $4^4 3^{15} 2^{30-i} 1^{3i}$, where $0\le i\le 30$; 
    \item $4^4 3^{14} 2^{31-i} 1^{4+3i}$, where $0\le i\le 31$; 
    \item $4^4 3^{13} 2^{33-i} 1^{5+3i}$, where $0\le i\le 33$; 
    \item $4^4 3^{12} 2^{37-i} 1^{3i}$, where $0\le i\le 37$; 
    \item $4^4 3^{11} 2^{38-i} 1^{4+3i}$, where $0\le i\le 38$; 
    \item $4^4 3^{10} 2^{40-i} 1^{5+3i}$, where $0\le i\le 40$; 
    \item $4^4 3^9 2^{44-i} 1^{3i}$, where $0\le i\le 44$; 
    \item $4^4 3^8 2^{45-i} 1^{4+3i}$, where $0\le i\le 45$; 
    \item $4^4 3^7 2^{47-i} 1^{5+3i}$, where $0\le i\le 47$; 
    \item $4^4 3^6 2^{51-i} 1^{3i}$, where $0\le i\le 51$; 
    \item $4^4 3^5 2^{52-i} 1^{4+3i}$, where $0\le i\le 52$; 
    \item $4^4 3^4 2^{54-i} 1^{5+3i}$, where $0\le i\le 54$; 
    \item $4^4 3^3 2^{58-i} 1^{3i}$, where $0\le i\le 58$; 
    \item $4^4 3^2 2^{59-i} 1^{4+3i}$, where $0\le i\le 59$; 
    \item $4^4 3^1 2^{61-i} 1^{5+3i}$, where $0\le i\le 61$; 
    \item $4^4 3^0 2^{65-i} 1^{3i}$, where $0\le i\le 65$; 
    \item $4^3 3^{30}$; 
    \item $4^3 3^{29}  2^{1-i} 1^4+3i$, where $0\le i\le 1$; 
    \item $4^3 3^{28} 2^{2-i} 1^{8+3i}$, where $0\le i\le 2$; 
    \item $4^3 3^{27} 2^{7-i} 1^{3i}$, where $0\le i\le 7$; 
    \item $4^3 3^{26} 2^{8-i} 1^{4+3i}$, where $0\le i\le 8$; 
    \item $4^3 3^{25} 2^{10-i} 1^{5+3i}$, where $0\le i\le 10$; 
    \item $4^3 3^{24} 2^{14-i} 1^{3i}$, where $0\le i\le 14$; 
    \item $4^3 3^{23} 2^{15-i} 1^{4+3i}$, where $0\le i\le 15$; 
    \item $4^3 3^{22} 2^{17-i} 1^{5+3i}$, where $0\le i\le 17$; 
    \item $4^3 3^{21} 2^{21-i} 1^{3i}$, where $0\le i\le 21$; 
    \item $4^3 3^{20} 2^{22-i} 1^{4+3i}$, where $0\le i\le 22$; 
    \item $4^3 3^{19} 2^{24-i} 1^{5+3i}$, where $0\le i\le 24$; 
    \item $4^3 3^{18} 2^{28-i} 1^{3i}$, where $0\le i\le 2$; 
    \item $4^3 3^{17} 2^{29-i} 1^{4+3i}$, where $0\le i\le 29$; 
    \item $4^3 3^{16} 2^{31-i} 1^{5+3i}$, where $0\le i\le 31$; 
    \item $4^3 3^{15} 2^{35-i} 1^{3i}$, where $0\le i\le 35$; 
    \item $4^3 3^{14} 2^{36-i} 1^{4+3i}$, where $0\le i\le 36$; 
    \item $4^3 3^{13} 2^{38-i} 1^{5+3i}$, where $0\le i\le 38$; 
    \item $4^3 3^{12} 2^{42-i} 1^{3i}$, where $0\le i\le 42$; 
    \item $4^3 3^{11} 2^{43-i} 1^{4+3i}$, where $0\le i\le 43$; 
    \item $4^3 3^{10} 2^{45-i} 1^{5+3i}$, where $0\le i\le 45$; 
    \item $4^3 3^9 2^{49-i} 1^{3i}$, where $0\le i\le 49$; 
    \item $4^3 3^8 2^{50-i} 1^{4+3i}$, where $0\le i\le 50$; 
    \item $4^3 3^7 2^{52-i} 1^{5+3i}$, where $0\le i\le 52$; 
    \item $4^3 3^6 2^{56-i} 1^{3i}$, where $0\le i\le 56$; 
    \item $4^3 3^5 2^{57-i} 1^{4+3i}$, where $0\le i\le 57$; 
    \item $4^3 3^4 2^{59-i} 1^{5+3i}$, where $0\le i\le 59$; 
    \item $4^3 3^3 2^{63-i} 1^{3i}$, where $0\le i\le 63$; 
    \item $4^3 3^2 2^{64-i} 1^{4+3i}$, where $0\le i\le 64$; 
    \item $4^3 3^1 2^{66-i} 1^{5+3i}$, where $0\le i\le 66$; 
    \item $4^3 3^0 2^{70-i} 1^{3i}$, where $0\le i\le 70$; 
    \item $4^2 3^{31}  1^8$; 
    \item $4^2 3^{30} 2^{5-i} 1^{3i}$, where $0\le i\le 5$; 
    \item $4^2 3^{29} 2^{6-i} 1^{4+3i}$, where $0\le i\le 6$; 
    \item $4^2 3^{28} 2^{8-i} 1^{5+3i}$, where $0\le i\le 8$; 
    \item $4^2 3^{27} 2^{12-i} 1^{3i}$, where $0\le i\le 12$; 
    \item $4^2 3^{26} 2^{13-i} 1^{4+3i}$, where $0\le i\le 13$; 
    \item $4^2 3^{25} 2^{15-i} 1^{5+3i}$, where $0\le i\le 15$; 
    \item $4^2 3^{24} 2^{19-i} 1^{3i}$, where $0\le i\le 19$; 
    \item $4^2 3^{23} 2^{20-i} 1^{4+3i}$, where $0\le i\le 20$; 
    \item $4^2 3^{22} 2^{22-i} 1^{5+3i}$, where $0\le i\le 22$; 
    \item $4^2 3^{21} 2^{26-i} 1^{3i}$, where $0\le i\le 26$; 
    \item $4^2 3^{20} 2^{27-i} 1^{4+3i}$, where $0\le i\le 27$; 
    \item $4^2 3^{19} 2^{29-i} 1^{5+3i}$, where $0\le i\le 29$; 
    \item $4^2 3^{18} 2^{33-i} 1^{3i}$, where $0\le i\le 33$; 
    \item $4^2 3^{17} 2^{34-i} 1^{4+3i}$, where $0\le i\le 34$; 
    \item $4^2 3^{16} 2^{36-i} 1^{5+3i}$, where $0\le i\le 36$; 
    \item $4^2 3^{15} 2^{40-i} 1^{3i}$, where $0\le i\le 40$; 
    \item $4^2 3^{14} 2^{41-i} 1^{4+3i}$, where $0\le i\le 41$; 
    \item $4^2 3^{13} 2^{43-i} 1^{5+3i}$, where $0\le i\le 43$; 
    \item $4^2 3^{12} 2^{47-i} 1^{3i}$, where $0\le i\le 47$; 
    \item $4^2 3^{11} 2^{48-i} 1^{4+3i}$, where $0\le i\le 48$; 
    \item $4^2 3^{10} 2^{50-i} 1^{5+3i}$, where $0\le i\le 50$; 
    \item $4^2 3^9 2^{54-i} 1^{3i}$, where $0\le i\le 54$; 
    \item $4^2 3^8 2^{55-i} 1^{4+3i}$, where $0\le i\le 55$; 
    \item $4^2 3^7 2^{57-i} 1^{5+3i}$, where $0\le i\le 57$; 
    \item $4^2 3^6 2^{61-i} 1^{3i}$, where $0\le i\le 61$; 
    \item $4^2 3^5 2^{62-i} 1^{4+3i}$, where $0\le i\le 62$; 
    \item $4^2 3^4 2^{64-i} 1^{5+3i}$, where $0\le i\le 64$; 
    \item $4^2 3^3 2^{68-i} 1^{3i}$, where $0\le i\le 68$; 
    \item $4^2 3^2 2^{69-i} 1^{4+3i}$, where $0\le i\le 69$; 
    \item $4^2 3^1 2^{71-i} 1^{5+3i}$, where $0\le i\le 71$; 
    \item $4^2 3^0 2^{75-i} 1^{3i}$, where $0\le i\le 75$; 
    \item $4^1 3^{32} 2^{4-i} 1^{4+3i}$, where $0\le i\le 4$; 
    \item $4^1 3^{31} 2^{6-i} 1^{5+3i}$, where $0\le i\le 6$; 
    \item $4^1 3^{30} 2^{10-i} 1^{3i}$, where $0\le i\le 10$; 
    \item $4^1 3^{29} 2^{11-i} 1^{4+3i}$, where $0\le i\le 11$; 
    \item $4^1 3^{28} 2^{13-i} 1^{5+3i}$, where $0\le i\le 13$; 
    \item $4^1 3^{27} 2^{17-i} 1^{3i}$, where $0\le i\le 17$; 
    \item $4^1 3^{26} 2^{18-i} 1^{4+3i}$, where $0\le i\le 18$; 
    \item $4^1 3^{25} 2^{20-i} 1^{5+3i}$, where $0\le i\le 20$; 
    \item $4^1 3^{24} 2^{24-i} 1^{3i}$, where $0\le i\le 24$; 
    \item $4^1 3^{23} 2^{25-i} 1^{4+3i}$, where $0\le i\le 25$; 
    \item $4^1 3^{22} 2^{27-i} 1^{5+3i}$, where $0\le i\le 27$; 
    \item $4^1 3^{21} 2^{31-i} 1^{3i}$, where $0\le i\le 31$; 
    \item $4^1 3^{20} 2^{32-i} 1^{4+3i}$, where $0\le i\le 32$; 
    \item $4^1 3^{19} 2^{34-i} 1^{5+3i}$, where $0\le i\le 34$; 
    \item $4^1 3^{18} 2^{38-i} 1^{3i}$, where $0\le i\le 38$; 
    \item $4^1 3^{17} 2^{39-i} 1^{4+3i}$, where $0\le i\le 39$; 
    \item $4^1 3^{16} 2^{41-i} 1^{5+3i}$, where $0\le i\le 41$; 
    \item $4^1 3^{15} 2^{45-i} 1^{3i}$, where $0\le i\le 45$; 
    \item $4^1 3^{14} 2^{46-i} 1^{4+3i}$, where $0\le i\le 46$; 
    \item $4^1 3^{13} 2^{48-i} 1^{5+3i}$, where $0\le i\le 48$; 
    \item $4^1 3^{12} 2^{52-i} 1^{3i}$, where $0\le i\le 52$; 
    \item $4^1 3^{11} 2^{53-i} 1^{4+3i}$, where $0\le i\le 53$; 
    \item $4^1 3^{10} 2^{55-i} 1^{5+3i}$, where $0\le i\le 55$; 
    \item $4^1 3^9 2^{59-i} 1^{3i}$, where $0\le i\le 59$; 
    \item $4^1 3^8 2^{60-i} 1^{4+3i}$, where $0\le i\le 60$; 
    \item $4^1 3^7 2^{62-i} 1^{5+3i}$, where $0\le i\le 62$; 
    \item $4^1 3^6 2^{66-i} 1^{3i}$, where $0\le i\le 66$; 
    \item $4^1 3^5 2^{67-i} 1^{4+3i}$, where $0\le i\le 67$; 
    \item $4^1 3^4 2^{69-i} 1^{5+3i}$, where $0\le i\le 69$; 
    \item $4^1 3^3 2^{73-i} 1^{3i}$, where $0\le i\le 73$; 
    \item $4^1 3^2 2^{74-i} 1^{4+3i}$, where $0\le i\le 74$; 
    \item $4^1 3^1 2^{76-i} 1^{5+3i}$, where $0\le i\le 76$; 
    \item $4^1 3^0 2^{80-i} 1^{3i}$, where $0\le i\le 80$; 
    \item $4^0 3^{34} 2^{3-i} 1^{8+3i}$, where $0\le i\le 3$; 
    \item $4^0 3^{33} 2^{8-i} 1^{3i}$, where $0\le i\le 8$; 
    \item $4^0 3^{32} 2^{9-i} 1^{4+3i}$, where $0\le i\le 9$; 
    \item $4^0 3^{31} 2^{11-i} 1^{5+3i}$, where $0\le i\le 11$; 
    \item $4^0 3^{30} 2^{15-i} 1^{3i}$, where $0\le i\le 15$; 
    \item $4^0 3^{29} 2^{16-i} 1^{4+3i}$, where $0\le i\le 16$; 
    \item $4^0 3^{28} 2^{18-i} 1^{5+3i}$, where $0\le i\le 18$; 
    \item $4^0 3^{27} 2^{22-i} 1^{3i}$, where $0\le i\le 22$; 
    \item $4^0 3^{26} 2^{23-i} 1^{4+3i}$, where $0\le i\le 23$; 
    \item $4^0 3^{25} 2^{25-i} 1^{5+3i}$, where $0\le i\le 25$; 
    \item $4^0 3^{24} 2^{29-i} 1^{3i}$, where $0\le i\le 29$; 
    \item $4^0 3^{23} 2^{30-i} 1^{4+3i}$, where $0\le i\le 30$; 
    \item $4^0 3^{22} 2^{32-i} 1^{5+3i}$, where $0\le i\le 32$; 
    \item $4^0 3^{21} 2^{36-i} 1^{3i}$, where $0\le i\le 36$; 
    \item $4^0 3^{20} 2^{37-i} 1^{4+3i}$, where $0\le i\le 37$; 
    \item $4^0 3^{19} 2^{39-i} 1^{5+3i}$, where $0\le i\le 39$; 
    \item $4^0 3^{18} 2^{43-i} 1^{3i}$, where $0\le i\le 43$; 
    \item $4^0 3^{17} 2^{44-i} 1^{4+3i}$, where $0\le i\le 44$; 
    \item $4^0 3^{16} 2^{46-i} 1^{5+3i}$, where $0\le i\le 46$; 
    \item $4^0 3^{15} 2^{50-i} 1^{3i}$, where $0\le i\le 50$; 
    \item $4^0 3^{14} 2^{51-i} 1^{4+3i}$, where $0\le i\le 51$; 
    \item $4^0 3^{13} 2^{53-i} 1^{5+3i}$, where $0\le i\le 53$; 
    \item $4^0 3^{12} 2^{57-i} 1^{3i}$, where $0\le i\le 57$; 
    \item $4^0 3^{11} 2^{58-i} 1^{4+3i}$, where $0\le i\le 58$; 
    \item $4^0 3^{10} 2^{60-i} 1^{5+3i}$, where $0\le i\le 60$; 
    \item $4^0 3^9 2^{64-i} 1^{3i}$, where $0\le i\le 64$; 
    \item $4^0 3^8 2^{65-i} 1^{4+3i}$, where $0\le i\le 65$; 
    \item $4^0 3^7 2^{67-i} 1^{5+3i}$, where $0\le i\le 67$; 
    \item $4^0 3^6 2^{71-i} 1^{3i}$, where $0\le i\le 71$; 
    \item $4^0 3^5 2^{72-i} 1^{4+3i}$, where $0\le i\le 72$; 
    \item $4^0 3^4 2^{74-i} 1^{5+3i}$, where $0\le i\le 74$; 
    \item $4^0 3^3 2^{78-i} 1^{3i}$, where $0\le i\le 78$; 
    \item $4^0 3^2 2^{79-i} 1^{4+3i}$, where $0\le i\le 79$; 
    \item $4^0 3^1 2^{81-i} 1^{5+3i}$, where $0\le i\le 81$; 
    \item $4^0 3^0 2^{85-i} 1^{3i}$, where $0\le i\le 85$; 
  \end{itemize}
\end{proposition}

\end{document}